\documentclass{article}
\usepackage{epstopdf}
\usepackage{amssymb}
\usepackage{amsmath}
\usepackage{amsthm}
\usepackage{amsbsy}
\usepackage{psfrag}
\usepackage{pstricks}
\usepackage{graphics}
\usepackage{graphicx}
\usepackage{graphicx,amsmath,amsthm,wasysym,amssymb,mathrsfs}
\usepackage{mathtools,appendix,epsfig,epsf,color,subfigure}
\usepackage{hyperref,pgfplots}
\usepackage{pgfplotstable}
\usepackage{setspace,relsize,needspace,etoolbox}
\usepackage[section]{placeins}
\usepackage{lineno} \linenumbers
\makeatletter
\preto{\@verbatim}{\topsep=-1.5pt \partopsep=-1pt }
\makeatother
\usepackage[morefloats=7]{morefloats}
\theoremstyle{plain}
\newtheorem{theorem}{Theorem}
\newtheorem{lemma}[theorem]{Lemma}

\newtheorem{example}{Example}
\theoremstyle{remark}

\title{Total Variation Diminishing (TVD) method for Elastohydrodynamic Lubrication (EHL) problem on Parallel Computers}
\author{Peeyush Singh$^1$ and Pravir Dutt$^2$\\
  \small \noindent $^1$Vellore Institute of Technology-AP, University\\
  \small \noindent Department of Mathematics, Andhra Pradesh-522237, India\\  
  \small \noindent $^1$E-mail: peeyush.singh@vitap.ac.in\\
 \small \noindent $^2$ Department of Mathematics and Statistics, IIT Kanpur-208016, India.\\
 \small \noindent $^2$E-mail: pravir@iitk.ac.in
}
\date{}
\begin{document}
\maketitle

\abstract{In this article, we offer a novel parallel approach for the solution of elastohydrodynamic lubrication line and point contact problems using a class of total variation diminishing (TVD) schemes on parallel computers. A direct parallel approach is presented by introducing a novel solver named as projected alternate quadrant interlocking factorization (PAQIF) by solving discrete variational inequality.
For one-dimensional EHL case, we use weighted change in Newton-Raphson approximation to compute the Jacobian matrix in the form of a banded matrix by dividing two subregions on the whole computation domain.
Such subregion matrices are then assembled by measuring the ratio of diffusive coefficient and discrete grid length on the domain of the interest.
The banded matrix is then processed to parallel computers for solving discrete linearized complementarity system using PAQIF algorithm.
The idea is  easily extended in two-dimensional EHL case by taking appropriate splitting in $x$ and $y$ alternating directions respectively.  
Numerical experiments are performed and analyzed to validate the performance of computed solution on serial and parallel computers.\\
}

\noindent {\sc Keywords:} \ TVD schemes, projected alternate quadrant interlocking factorization (PAQIF) , Variational inequality, Elastohydrodynamic Lubrication, parallel computers.

\vspace{0.2in}
\newpage
\section{Introduction}\label{sec:one}
Elastohydrodynamic Lubrication (EHL) problems had been studied by many researchers in last several decades. A milestone numerical computation on EHL are categorized by the authors
(e.g. \cite{sahmed14},\cite{cimatti},\cite{ehlbook},\cite{v1},\cite{vr},\\ \cite{hamrock},\cite{ALubrecht},  \cite{peeyush,peeyush2020},\cite{moes},
\cite{venner94},\cite{holmes},\cite{hlu},\cite{habchi}).  
In 1992, Venner \cite{vr} has introduced a low order discretization for EHL model (see ~\ref{model:ehl}) using multi-grid and multi-level multi-integration approach
which is stable for larger range of load parameters. There are few other independent work also have been noticed by the authors e.g. differential deflection method 
by Cardiff group \cite{holmes}, Discontinuous Galerkin method by Leeds group \cite{hlu} and FEM-based Newton method by INSA de Lyon group \cite{habchi} (However, 
in this case, the deformation is modeled in PDE form ) etc. In 2013, a review work is presented by Lugt et al. \cite{lugt} provide a rigorous detail on the current
EHL development activities in the field.
Recently, Peeyush et al. \cite{peeyush,peeyush2019} extended Venner idea into a class of total variation diminishing (TVD) approach by producing a class of splittings.\\
Although there are several numerical works are presented in solving EHL problem on serial computer, application on parallel computation in this area is quite few see for example \cite{berzins2007,Vazquez2000}.
Continuing in this direction, this article is devoted in numerical study of EHL problem using parallel computation. 
In 1999, S.C. S. Rao \cite{pravir} introduced a direct parallel solution of the banded linear system by an alternate quadrant interlocking factorization (AQIF) algorithm which is different from Gaussian
elimination algorithm as factor matrices are not triangular.
He also proved that AQIF algorithm is stable for symmetric and diagonally dominated matrices (i.e.free from any blow up) and solve almost independently on parallel computers.
Furthermore, in spite of its large complexity, the substantial speedup of algorithm, when implemented on parallel processor remains high.
This is the main motivation for present study to adopt PAQIF algorithm using total variation diminishing (TVD) approach for the EHL model problem.
Therefore, in this article an attempt has been made to develop a novel solver for EHL problem generalizing TVD concept efficiently.\\
The concept of TVD has been established by Harten and later by Sweby \cite{harten83},\cite{harten84},\cite{sweby} to avoid unphysical wiggles in a numerical scheme.
Harten also has given necessary and sufficient condition for a scheme to be TVD. To understand the concept, we first define 
the notation total variation $TV$ of a mesh function $u^{n}$ as
\begin{align}
\label{eqn1}
 TV(u^{n}) = \displaystyle\sum_{-\infty}^{\infty}|u_{j+1}^{n}-u_{j}^{n}|=\displaystyle\sum_{-\infty}^{\infty}|\Delta_{j+1/2}u^{n}|
\end{align}
having the following convention 
\begin{align}
\label{eqn2}
 \Delta_{j+1/2}u^{n} = u_{j+1}^{n}-u_{j}^{n}
\end{align}
for any mesh function $u$ is used.
Harten's theory is understood in the form of conservation laws 
\begin{align}
\label{eqn3}
 u_{t}+ f(u)_{x} = 0.
\end{align}
The numerical approximation of Eq.~(\ref{eqn3}) is said to be TVD if
\begin{align}
\label{eqn4}
 TV(u^{n+1}) \le TV(u^{n})
\end{align}
Then Harten's condition for any scheme to be TVD is explained below.
\begin{theorem}Let a general numerical scheme for conservation laws Eq.~(\ref{eqn3}) is of the form
 \begin{gather}
 \label{eqn5}
 u^{n+1}_{i}=u^{n}_{i}-c_{i}^{n}(u_{i}^{n}-u_{i-1}^{n})+d_{i}^{n}(u_{i+1}^{n}-u_{i}^{n})
 \end{gather}
 over one time step, where the coefficients $c_{i}^{n}$ and $d_{i}^{n}$ are arbitrary value (In
 practice it may depend on values $u^{n}_{i}$ in some way i.e., the method may be nonlinear).
 Then $TV(u^{n+1}) \leq TV(u^{n})$ provided  the following conditions are satisfied
 \begin{gather}
 \label{eqn6}
 c^{n}_{i} \geq 0 \quad ,
 d^{n}_{i} \geq 0\quad ,
 c^{n}_{i}+d^{n}_{i} \leq 1\quad \forall i
  \end{gather}
 \end{theorem}
There has been a very well developed TVD theory available in literature for time dependent problem.
Additionally, this concept is also extended for steady state convection-diffusion case in the form of $M$- matrix \cite{Varga} 
using appropriate flux limiting schemes \cite{koren},\cite{koren88},\cite{Oosterlee},\cite{osterlee2003}. However, very little attention have been paid
in developing TVD schemes for EHL problems. In this article, our aim to investigate a class of splitting for EHL model which is robust and
high order accurate ( at least second order in smooth part of the solution ) for larger range of load parameters.
The rest of the article is organized as followed. In Section.~\ref{section:s2}, few preliminaries are discussed about the parallel PAQIF algorithm and complexity of the
algorithm.
In Section~\ref{sec:three}, a series of splitting are constructed by imitating linear convection-diffusion model for applying PAQIF algorithm to solve the EHL model defined in Example 5 and Example 6. The convergence analysis of the splittings is also discussed.
In Section~\ref{sec:six}, numerical experiments are conducted to check the performance of present splitting and its improvement to the EHL model.
At the end of Section~\ref{sec:seven}, overall conclusion is summarized.  
\section{Preliminaries}\label{section:s2}
We first consider partitioning of the linear complementarity system, then to decouple the partitioned linear sub-complementarity system we introduce PAQIF
and finally discuss the present method.\\
\subsection{Partitioning of the Linear Complementarity System}\label{subsection:subs1}
Consider the linear complementarity problem 
\begin{align}
 L U(x) \ge f(x) \quad    x \in \Omega \nonumber \\
 U(x) \ge 0      \quad    x \in \Omega \nonumber \\
 U(x)^{T}.[L U(x)-f(x)]=0 \quad    x \in \Omega \nonumber \\ \label{eqn14}
 U(x) = g(x)   \quad    x \in \partial \Omega 
\end{align}
We now subdivide the linear complementarity problem into $r$ blocks linear sub-complementarity problem each of size $n$ along the main diagonal such that $N=nr$,
where $r$ is the number of processors available.
The linear complementarity problem Eqn.~\ref{eqn14} is partitioned into 
\begin{align}\label{eqn15}
 L^{(m)}_{-}U^{(m-1)}(x)+L^{(m)}_{0}U^{(m)}(x)+L^{(m)}_{+}U^{(m+1)}(x) \ge f^{(m)}(x), \quad m=1,2,.,r \\
 U^{(m)}(x) \ge 0 \\
U^{(m)}(x)^{T}.(L^{(m)}_{-}U^{(m-1)}(x)+L^{(m)}_{0}U^{(m)}(x)+L^{(m)}_{+}U^{(m+1)}(x)-f^{(m)}(x))=0,
\end{align}
where $L^{(m)}_{0}$ is the $n\times n$  block diagonal coefficient matrix of each partition, $L^{(m)}_{-}$ and $L^{(m)}_{+}$ are $n \times n$ accompanied left and right block matrices.
$U^{(m-1)}(x),U^{(m)}(x),U^{(m+1)}(x)$ and $f^{(m)}(x)$ are $n \times 1$ vectors.
\begin{align*}
 L^{(1)}_{-}= O_{n \times n}, \quad L^{(r)}_{+} = O_{n \times n};\quad U^{(0)}(x)=0,\quad U^{(r+1)}(x)=0 ;
 \end{align*}
\[
L^{(m)}_{+}=
\left[
\begin{array}{c|c}
0 & 0 \\
\hline
L_{{m}_{+}} & 0
\end{array}
\right]_{n \times n}
,
L^{(m)}_{-}=
\left[
\begin{array}{c|c}
0 & L_{{m}_{-}} \\
\hline
0 & 0
\end{array}
\right]_{n \times n};
\]
\begin{align*}
U^{m}(x)=\Big[U^{(m)}(x_{1}),...,U^{(m)}(x_{n}) \Big]^{T},
f^{m}(x)=\Big[f^{(m)}(x_{1}),...,f^{(m)}(x_{n}) \Big]^{T}
\end{align*}
$L^{(m)}_{-}$ and $L^{(m)}_{+}$ are upper and lower triangular matrices, respectively. For each partition $r$, 
Eqn.~\ref{eqn15} can be reformulated as
\begin{align}\label{eqnn16}
L^{(m)}_{0}U^{(m)}(x) \ge f^{(m)}(x)-\left[
\begin{array}{c}
L^{(m)}_{-}U^{(m-1)}_{L}(x)  \\
0          \\
.          \\
.          \\
0          \\
L^{(m)}_{+}U^{(m+1)}_{F}(x)
\end{array}
\right]_{n \times 1} :=f^{*(m)}(x), \quad m=1,..,r \\
U^{(m)}(x) \ge 0 \\
U^{(m)}(x)^{T}.(L^{(m)}_{0}U^{(m)}(x)-f^{*(m)}(x))=0,
\end{align}
where $U^{(m-1)}_{L}(x)$ and $U^{(m+1)}_{F}(x)$ are $\beta_{v} \times 1$ vectors picked up from the last and first $\beta_{v}$ components of the solution
vector $U^{(m-1)}(x)$ and $U^{(m+1)}(x)$, respectively.
Now, in order to decouple the sub-complementarity problem in Eqn.~\ref{eqnn16}, so that they can processed in parallel, we first note the fact that in Eqn.~\ref{eqnn16}
$f^{*(m)}(x)$ differs from $f^{(m)}(x)$ only in its first $\beta_{v}$ and last $\beta_{v}$ components. In order to factorize $L^{(m)}_{0}$ into $W^{(m)}_{0}Z^{(m)}_{0}$,
we consider the space generated by $e_{i},e_{n-i+1};1 \le i \le \beta_{v}$ (i.e. $\textit{span}_{1 \le i \le \beta_{v}}\{e_{i},e_{n-i+1}\}$) is invariant under the transformation
$W^{(m)}_{0}$ (and so  invariant under its inverse transformation ${W^{(m)}}^{-1}_{0}$), where $e_{j}:= (0,0,..,0,1_{j^{th} \textit{term}},0,..,0)$. Furthermore, the solution procedure with the
matrix $Z^{(m)}_{0}$ moves from the first and last unknowns towards middle one.
\subsection{Projected Alternate Quadrant Interlocking Factorization}\label{subsection:subs2}
This factorization is highly motivated by pioneer work of Rao \cite{pravir} on AQIF and it is proved that the method is stable for nonsingular diagonally dominant matrices.
PAQIF method has mild change in its procedure as projection is incorporated on convex set during computation.
The element $W^{(m)}_{0}$ and $Z^{(m)}_{0}$ are given by the relations 
\begin{equation}
    w_{i,j}=
    \begin{cases}
      1,& i= j \\
      0,& \forall \quad 1 \leqslant j \leqslant [n/2], (j+1) \leqslant i \leqslant (n-j+1)  \\
      0,& \forall\quad n+1-[n/2] \leqslant j \leqslant n, n-j+1 \leqslant i \leqslant j-1 \\
      w_{i,j}, & \textit{otherwise;}
    \end{cases}
  \end{equation}
\begin{equation}
    z_{i,j}=
    \begin{cases}
      0,& \forall \quad 1 \leqslant i \leqslant [(n-1)/2], (i+1) \leqslant j \leqslant (n-i) \\
      0,& \forall\quad n+1-[n/2] \leqslant i \leqslant n, n-i+2 \leqslant j \leqslant i-1 \\
      w_{i,j}, & \textit{otherwise;}
    \end{cases}
  \end{equation}
where the symbol $[m]$ means for largest integer $\leqslant m$, $w_{i,j}$  and $z_{i,j}$  signify $(i,j)$th position elements of $W_{0}$ and $Z_{0}$ respectively.
Here we introduce PAQIF for general matrix and exposition of banded matrix is treated as special case.
\subsubsection{The Factorization}
Let $L_{0}^{(m)}$ be an even order matrix (say $n=2s$). Assume that there exist $W_{0}$ and $Z_{0}$ matrices such that 
\[L_{0}=W_{0}Z_{0},\]
where 
\begin{align*}
W_{0}= 
\begin{bmatrix}
1 & w_{1,2} & . & . &.     & .  & w_{1,n-1}  &0\\
  & 1 & w_{2,3} & . &.     &w_{2,n-2}   & 0  & \\
  &   & 1 & .   &.  &0     &    &      \\
  &   &         & 1 &0     &   &    & \\
  &   &         & 0 &1     &   &    & \\
  &   & 0       &   &    . &1  &    & \\
  & 0 &         & . &    . & w_{n-1,n-2} &  1 & \\
0 & w_{n,2}  &         & . &.     &.  & w_{n,n-1}  &1  
\end{bmatrix}
\end{align*}
\begin{align*}
Z_{0}= 
\begin{bmatrix}
z_{1,1} &  &   &   &     &                             &              &z_{1,n}\\
.  &z_{2,2}  &    &  &     &                           & z_{2,n-1}    &. \\
.  &  . & z_{3,3} &    &  & z_{3,n-2}                  & .            &. \\
.  &  . &  .       & z_{s,s} &z_{s,s+1}    &.          & .            &. \\
.  &  . &   .      & z_{s+1,s} &z_{s+1,s+1}    &.          & .            &. \\
.  &  . & z_{n-2,3}       &   &            &z_{n-2,n-2}    & .            &. \\
.  & z_{n-1,2}    &         &    &         &        &  z_{n-1,n-1}     &. \\
z_{n,1} &   &    &  &     &  &                                 &  z_{n,n}
\end{bmatrix}
\end{align*}
\subsubsection{Solution of the complementarity problem}
The solution of complementarity problem in Eqn. is obtained by solving two alternate systems
\begin{align*}
 W_{0}Y=F \textit{ and } Z_{0}U=Y
\end{align*}
and then projecting the computed solution $U$ on convex set $K$, where \[K=\{U_{0}| U_{0} \textit{ is solution of } L_{0}X= F \textit{ and } U_{0} \ge 0 \}.\]
In order to solve $W_{0}Y=F$, assume 
\begin{align*}
 W_{0}Y=F=F^{(1)}
\end{align*}
\[ 
\left. \begin{array}{r} 
y_{s-k+1}=b^{(k)}_{s-k+1}\\[1ex]
y_{s+k}=b^{(k)}_{s+k}
\end{array} \right\} 
\]
where 
\[b^{(k)}=b^{(k-1)}-y_{s-k+2}w_{s-k+2}-y_{s+k-1}w_{s+k-1},\quad 2 \le k \le s-1.\]
Also to solve another system $Z_{0}U=Y$, we perform the following steps.
At the kth $(1 \le k \le s)$ level we have to compute $2\times2$ system.
\[ 
\left. \begin{array}{r} 
z_{k,k}x_{k}+z_{k,k}x_{n-k+1}=y^{(k)}_{k}\\[1ex]
z_{n-k+1,k}x_{k}+z_{n-k+1,n-k+1}x_{n-k+1}=y^{(k)}_{n-k+1}
\end{array} \right\} 
\]
where
\[y^{(1)}=y\]
and 
\[y^{(k)}=y^{(k-1)}-x_{k}z_{k}-x_{n-k+1}z_{n-k+1},\quad 2 \le k \le s-1.\]
When $L_{0}$ is a banded matrix then PAQIF $W_{0}$ and $Z_{0}$ is rewritten as below.
\begin{align*}
W_{0}= 
\begin{bmatrix}
1 & w_{1,2} &.. w_{1,\beta_{v}} &  &     & w_{1,n-\beta_{v}}..  & w_{1,n-1}  &0\\
  & 1 & . & ..w_{s-\beta_{v}+1,s} &w_{s-\beta_{v}+1,s+1}..     &.   & 0  & \\
  &   & 1 & .   &.  &0     &    &      \\
  &   &         & 1 &0     &   &    & \\
  &   &         & 0 &1     &   &    & \\
  &   & 0       & w_{{s+2,s}_{:}}  &    w_{{s+2,s+1}_{:}} &1  &    & \\
  & 0 &         & ..w_{s+\beta_{v}+1,s} & w_{s+\beta_{v}+1,s+1}                                 & . .          &  1         & \\
0 & w_{n,2}  & ..w_{n,\beta_{v}+1}        &  &     &w_{n,n-\beta_{v}}..  & w_{n,n-1}  &1  
\end{bmatrix}
\end{align*}
\begin{align*}
Z_{0}= 
\begin{bmatrix}
z_{1,1} &  &   &   &     &                             &                            &z_{1,n}\\
:  &z_{2,2}  &    &  &     &                           & z_{2,n-1}                  &: \\
z_{\beta_{v}+1,1}  &  . & z_{3,3} &    &  & z_{3,n-2}                  & .            &z_{\beta_{v}+1,n} \\
   &  . &  .       & z_{s,s} &z_{s,s+1}    &.          & .                          & \\
   &  . &   .      & z_{s+1,s} &z_{s+1,s+1}    &.          & .                      & \\
z_{n-\beta_{v}+1,1}  &  . & z_{n-2,3}       &   &            &z_{n-2,n-2}    & .                      &z_{n-\beta_{v}+1,n} \\
:  & z_{n-1,2}    &         &    &         &        &  z_{n-1,n-1}                  &: \\
z_{n,1} &   &    &  &     &  &                                                      &  z_{n,n}
\end{bmatrix}
\end{align*}
\subsubsection{Evaluation of $W_{0}$ and $Z_{0}$ Matrices}
We illustrate at the outset of the kth level the matrix $L_{0}^{(k)}, \quad 1 \le k \le (s-1)$ with the components $l^{(k)}_{i,j}, 1 \le i,j \le n$ as detailed below.
\begin{align}
\left. \begin{array}{r}
 L_{0}^{(1)} = L_{0} \\
 L_{0}^{(k)} = L_{0}-\sum_{i=s-k+2}^{s}w_{i}z^{T}_{i}-\sum_{i=s+1}^{s+k-1}w_{i}z^{T}_{i}, \quad 2 \le k \le (s-1)
\end{array} \right\}  
\end{align}
whose central $(2k-2)$ rows and columns are zeros. We compute $s-k+1, s+k $ rows of $Z_{0}$ as \\
For
\begin{align}
(s-k-\beta_{v}+1) \le j \le (s-k+1), \nonumber \\
(s+k) \le j \le (s+k+\beta_{v}-1), \nonumber \\
 z_{s-k+1,j}=l^{(k)}_{s-k+1,j},
\end{align}
and for 
\begin{align}
(s-k-\beta_{v}+2) \le j \le (s-k+1), \nonumber \\
(s+k) \le j \le (s+k+\beta_{v}), \nonumber \\
 z_{s+k,j}=l^{(k)}_{s+k,j}.
\end{align}
Also we compute $s-k+1, s+k$ columns of $W_{0}$ as \\
For
\begin{align}
(s-k-\beta_{v}+1) \le i \le (s-k), \nonumber \\
(s+k+1) \le i \le (s+k+\beta_{v}), \nonumber\\
\left. \begin{array}{r}
z_{s-k+1,s-k+1}w_{i,s-k+1}+z_{s+k,s-k+1}w_{i,s+k} =l^{(k)}_{i,s-k+1} \\
z_{s-k+1,s+k}w_{i,s-k+1}+z_{s+k,s+k}w_{i,s+k} = l^{(k)}_{i,s+k}
\end{array} \right\}
\end{align}
Finally, we derive the matrix
\begin{align}
 L_{0}^{(k)} = L_{0}-w_{s-k+1}z^{T}_{s-k+1}-w_{s+k}z^{T}_{s+k}.
\end{align}
Finally for computing $z_{1,1},z_{1,n},z_{n,1}$ and $z_{n,n}$ elements of the matrix $Z_{0}$ for $k=s$, we have to perform (24) and (25).
\subsubsection{PAQIF method}
At this moment, we look at the solution of the complementarity system (15)-(17). This comprise of solving for $Y^{*(m)}$,
\begin{align}
W_{0}^{(m)}Y^{*(m)}= F^{*(m)}, \quad 1 \le m \le r,
\end{align}
and then computing for $U^{*(m)}$,
\begin{align}
Z_{0}^{(m)}U^{*(m)}= Y^{*(m)}, \quad 1 \le m \le r.
\end{align}
Let \[Y^{(m)}= [y_{1}^{m},...,y_{n}^{m}]^{T}\]
and consider 
\begin{align}
W_{0}^{(m)}Y^{(m)}= F^{(m)}, \quad 1 \le m \le r.
\end{align}
From the definition of $F^{*(m)}$ in Eqn (18), from Eqn (28) and Eqn (30) it deduces that
\begin{align}
 Y^{*(m)}=Y^{(m)}-\Big[ W^{(m)}_{0} \Big]^{-1}\left[
\begin{array}{c}
L^{(m)}_{-}U^{(m-1)}_{L}(x)  \\
0          \\
.          \\
.          \\
0          \\
L^{(m)}_{+}U^{(m+1)}_{F}(x)
\end{array}
\right]_{n \times 1}, 1 \le m \le r.
\end{align}
Once $Y^{(m)}$ are obtained from Eqn (30), the subsystem Eqn (29) may be rewritten as
\begin{align}
 Z_{0}^{(m)}U^{(m)}=Y^{(m)}-\Big[ W^{(m)}_{0} \Big]^{-1}\left[
\begin{array}{c}
L^{(m)}_{-}U^{(m-1)}_{L}(x)  \\
0          \\
.          \\
.          \\
0          \\
L^{(m)}_{+}U^{(m+1)}_{F}(x)
\end{array}
\right]_{n \times 1}, 1 \le m \le r.
\end{align}
Let the vectors $U^{m}$ and $Y^{m}$ be partitioned as below.
\[U^{(m)}=\left[\begin{array}{c}
U^{(m-1)}_{F}  \\
U^{(m-1)}_{M}         \\
U^{(m+1)}_{L}
\end{array}
\right],
Y^{(m)}=\left[\begin{array}{c}
Y^{(m-1)}_{F}  \\
Y^{(m-1)}_{M}         \\
Y^{(m+1)}_{L}
\end{array}
\right],\]
where \[U^{(m)}_{F}= [U^{(m)}_{1},..,U^{(m)}_{\beta_{v}}],U^{(m)}_{M}= [U^{(m)}_{\beta_{v}+1},..,U^{(m)}_{n-\beta_{v}}],\textit{ and }
U^{(m)}_{L}= [U^{(m)}_{n-\beta_{v}+1},..,U^{(m)}_{n}].\]
Let $Z_{0}^{(m)}$ be partitoned as
\begin{align}
 Z^{(m)}_{0}= 
 \begin{bmatrix}
  Z_{01}^{(m)} & 0              & Z_{02}^{(m)} \\
  Z_{05}^{(m)} &                & Z_{06}^{(m)} \\
  0            & Z^{*(m)}_{0}   & 0      \\
  Z_{07}^{(m)} &                & Z_{08}^{(m)} \\
  Z_{03}^{(m)} & 0              & Z_{04}^{(m)}
 \end{bmatrix},
\end{align}
where $Z_{0i}^{(m)}, 1\le i \le 8$ are $\beta_{v} \times \beta_{v}$ matrices and $Z_{0}^{*(m)}$ is an $(n-2\beta_{v})\times (n-2\beta_{v})$ matrix.
Let $\Big[W_{0}^{(m)}\Big]^{-1}$ be partitoned as below.
\begin{align}
 \Big[W^{(m)}_{0}\Big]^{-1}= 
 \begin{bmatrix}
  W_{01}^{(m)} & W_{05}^{(m)}              & W_{02}^{(m)} \\
  0            & W^{*(m)}_{0}   & 0                        \\
  W_{03}^{(m)} & W_{06}^{(m)}              & W_{04}^{(m)}
 \end{bmatrix},
\end{align}
where $W_{0i}^{(m)}, 1 \le i \le 4$ are $\beta_{v} \times \beta_{v}$ matrices, $W_{05}^{(m)},W_{06}^{(m)}$ are $\beta_{v} \times (n-2\beta_{v})$ matrices
and $W^{*(m)}_{0}$ is an $(n-2\beta_{v})\times(n-2\beta_{v})$ matrix (similar structure as that of $W^{(m)}_{0}$). We collect the first $\beta_{v}$ and last $\beta_{v}$
equation from each block in equation (33). These equations form a reduced system of order $2\beta_{v}r$ with semibandwidth $3\beta_{v}-1$, which is of the form
\begin{align}
 \begin{bmatrix}
  Z_{01}^{(m)} & Z_{02}^{(m)}  & C^{(1)}_{1}  &             &                         &               &              &              \\
  Z_{03}^{(m)} &Z_{03}^{(m)}   & C_{02}^{(1)} &             &                         &               &              &              \\
               & B^{(m)}_{1}   &Z^{(2)}_{01}  & Z^{(2)}_{02}&C_{02}^{(1)}             &               &              &               \\
               & B^{(m)}_{2}   & Z^{(m)}_{3 } & Z^{(m)}_{4} & C_{02}^{(2)}            &               &              &               \\
               &               &              & .           & .                       & .             &              &               \\
               &               &              &             & .                       & .             &.             &               \\
               &               &              &             &                         & .             &.             & .             \\
               &               &              &             &                         & Z_{08}^{(m)}  & Z_{08}^{(m)} & Z_{08}^{(m)} \\
               &               &              &             &                         & Z_{08}^{(m)}  & Z_{08}^{(m)} & Z_{04}^{(m)}
 \end{bmatrix}
 \left[\begin{array}{c}
U^{(1)}_{F}  \\
U^{(1)}_{L}  \\
U^{(2)}_{F}  \\
U^{(2)}_{L}  \\
.            \\
.            \\
.            \\
U^{(r)}_{F}  \\
U^{(r)}_{L}
\end{array}
\right]
=
\left[\begin{array}{c}
Y^{(1)}_{F}  \\
Y^{(1)}_{L}  \\
Y^{(2)}_{F}  \\
Y^{(2)}_{L}  \\
.            \\
.            \\
.            \\
Y^{(r)}_{F}  \\
Y^{(r)}_{L}
\end{array}
\right],
\end{align}
where \[\bar{B}_{1}^{(m)}=W_{1}^{(m)}\hat{B}_{1}^{(m)}, \bar{C}_{1}^{(m)}= W_{1}^{(m)}\hat{C}_{1}^{(m)},
\bar{B}_{2}^{(m)}= W_{1}^{(m)}\hat{B}_{1}^{(m)}, \bar{C}_{2}^{(m)}=W_{1}^{(m)}\hat{C}_{1}^{(m)}.\]
The reduced system defined in equation (35) can be represented as 
\begin{align}
 R_{d}U_{d}=F_{d}.
\end{align}
At this stage we form 
\begin{align}
 R_{d}^{T}R_{d}U_{d}=R_{d}^{T}F_{d}.
\end{align}
Since $R_{d}^{T}R_{d}$ is symmetric positive definite matrix and it can be solved using cholesky factorization method without use of any pivoting.
First system (37) has been solved for $U^{(m)}_{F},U^{(m)}_{L}, 1\le m \le r$ and then the computed solutions has been projected to the convex set $K$.
Now subsystem (32) is easily decoupled into
\begin{align}
Z^{*(m)}_{0}U^{(m)}_{M}=y^{(m)}-\left[\begin{array}{c}
Z^{(m)}_{05}U^{(m)}_{F}+Z^{(m)}_{05}U^{(m)}_{L}  \\
0                                     \\
Z^{(m)}_{07}U^{(m)}_{F}+Z^{(m)}_{08}U^{(m)}_{L}
\end{array}
\right], 1\le m \le r. 
\end{align}
Over all method is now outlined in brief as follows:\\
Step 1: For $m=1,2,..,r$ factorize in parallel 
\[L^{(m)}_{0} = W_{0}^{(m)}Z_{0}^{(m)}\]
Step 2: For $m=1,2,..,r$ compute $Y^{(m)}$ in parallel
\[W^{(m)}_{0}Y^{(m)} = F^{(m)}\]
Step 3: For $m=1,2,..,r$ get inverse of $2 \beta_{v} \times 2 \beta_{v}$ matrix obtained by collecting first $\beta_{v}$ and last $\beta_{v}$ rows and 
columns of $W_{0}^{(m)}$ in parallel.\\ \\
Step 4: Solve the reduced system from the subsystem (32) by collecting first $\beta_{v}$ and last $\beta_{v}$ equations from each block.
Then form normal equations (37). Solve system (37) for $U_{F}^{(m)}$ and $U_{L}^{(m)}, m=1,2,..,r$.\\ \\
Step 5: Project $U_{F}^{(m)}$ and $U_{L}^{(m)}, m=1,2,..,r$ into convex set $K$. \\ \\
Step 6: For $m=1,2,..,r$ solve $U_{M}^{(m)}$ in parallel from (38).\\ \\
Step 7: Project $U_{M}^{(m)}, m=1,2,..,r$ into convex set $K$.
\subsection{Complexity and speedup analysis of PAQIF}
In this section, we will discuss complexity and speedup analysis of PAQIF method.
\subsubsection{Serial complexity}
The serial count of the above algorithm is defined below.\\
{\bf Factorization count of $W_{0}$ and $Z_{0}$ matrices.} The number of execution steps required to compute the elements of $W_{0}$ and $Z_{0}$ matrices is given by,
\[ T_{\text{fact}}= r(t-1)(1 + 4\beta_{v} + 8\beta^{2}_{v})T_{{add}} + r(t-1)(2 + 8\beta_{v}+8\beta^{2}_{v})T_{\text{multi}} + r(t-1)4\beta_{v}T_{\text{div}}. \]
{\bf Time to calculate $Y$ elements.} The count of execution steps required to compute the elements $Y$ is given by,
\[ T_{\text{Y}}= r(t-1)4\beta_{v}T_{{add}} + r(t-1)4\beta_{v}T_{\text{multi}} \]
{\bf Time to calculate inversion matrices.}The number of execution cycles required to compute the inversions matrices is given by
\[T_{\text{inv}}= r4\beta^{3}_{v}T_{{add}} + r4\beta^{3}_{v}T_{\text{multi}}\]
{\bf Time to calculate accompanied matrices.}The number of execution cycles  required to compute the $\tilde{B}_{1},\tilde{B}_{2},\tilde{C}_{1},\tilde{C}_{2}$ matrices is given by
\[T_{\text{comp}}= r4\beta^{2}_{v}(\beta_{v}-1)T_{{add}} + r4\beta^{3}_{v}T_{\text{multi}}\]
Formation of normal equations requires
\[T_{\text{norm}}=(36r\beta^{3}_{v}-6r\beta^{2})T_{\text{add}}+(36r\beta^{3}_{v}-6r\beta^{2})T_{\text{multi}}.\]
Solution of normal equations by Cholesky factorization requires
\[T_{\text{chol}}=(36r\beta^{3}_{v}-6r\beta^{2}_{v}-10r\beta_{v})T_{\text{add}}+(36r\beta^{3}_{v}-6r\beta^{2}_{v}-10r\beta_{v})T_{\text{multi}}.\]
{\bf Time require to update $Y_{M}$.} For updating $Y_{M}$ requires
\[T_{\text{update}}=(r(t-1)(3+3\beta_{v})-r\beta_{v}(3+3\beta_{v}))T_{\text{add}}+ (r(t-1)(6+3\beta_{v})-r\beta_{v}(6+3\beta_{v}))T_{\text{multi}} \]
\subsubsection{Parallel complexity}
The parallel machine having $r$ processors operation count of the PAQIF algorithm are given below.\\

{\bf Factorization count of $W_{0}$ and $Z_{0}$ matrices.}\\
The number of parallel execution steps required to compute the elements of $W_{0}$ and $Z_{0}$ matrices is given by,
\[ T_{\text{fact}}= ( (t-1)(2 + 8\beta_{v}+8\beta^{2}_{v})T_{\text{op}}.\]
{\bf Time to calculate $Y$ elements.}\\
The count of execution steps required to compute the elements $Y$ is given by,
\[ T_{\text{Y}}= (t-1)4\beta_{v}T_{\text{op}}. \]
{\bf Time to calculate inversion matrices.}\\
The number of execution cycles required to compute the inversions matrices is given by
\[T_{\text{inv}}= 4\beta^{3}_{v}T_{\text{op}}\]
{\bf Time to calculate accompanied matrices.}\\
The number of execution cycles required to compute the $\tilde{B}_{1},\tilde{B}_{2},\tilde{C}_{1},\tilde{C}_{2}$ matrices is given by
\[T_{\text{comp}}= 4\beta^{3}_{v}T_{\text{op}}\]
Formation of normal equations requires
\[T_{\text{norm}}=(36\beta^{3}_{v}-6\beta^{2}_{v})T_{\text{op}}.\]
Solution of normal equations by Cholesky factorization requires
\[T_{\text{chol}}=(36\beta^{3}_{v}-6\beta^{2}_{v}-10\beta_{v})T_{\text{op}}.\]
{\bf Time require to update $Y_{M}$.}\\
For updating $Y_{M}$ requires
\[T_{\text{update}}=(2 + 2\beta_{v}^{2})T_{\text{op}} \]
and its solution requires
\[T_{\text{sol}}= r(t-\beta_{v}-1)T_{\text{op}}. \]
Overall algorithm requires $O(4\beta_{v}^{2}(N/r)+ \beta_{v}(11 + 9r) )$ time steps on an $`r'$ processor machine. Moreover, on a serial machine to solve banded linear system of size $N$ with semibandwidth $\beta_{v}$ requires $O(N\beta_{v}^{2})$ time steps. Consequently, speedup
\[ S_{p}= \frac{1}{4((1/r) + (\beta_{v}/N)(11+9r))}. \]
\subsubsection{Numerical experiment of PAQIF algorithm and its speedup performance}
All numerical computations are performed on Dell Tower precision having processor specification Intel(R) Core(TM) i7-6700 CPU @ 3.40GHz.
\begin{figure}[!htbp]
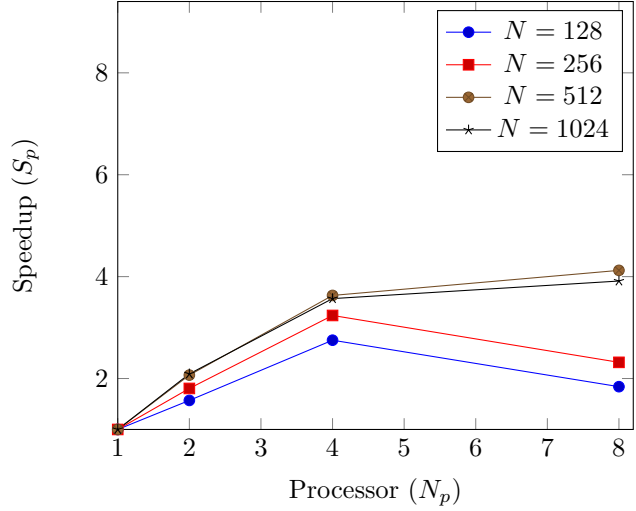

\centering
\tikzpicture
	\axis[
		xlabel=Processor ($N_{p}$),
		ylabel=Speedup ($S_{p}$),
		xmin=1, xmax = 8.2,
		ymin=1, ymax = 9.4,
		]
	\addplot coordinates {
		(1, 1)
		(2, 1.57)
		(4, 2.7521)
		(8, 1.8397)
	};
 	\addplot coordinates {
		(1,  1)
 		(2, 1.805)
 		(4, 3.2377)
		(8, 2.3174)
	};
 	\addplot coordinates {
		(1,  1)
 		(2, 2.067)
 		(4, 3.6307)
		(8, 4.1226)
	};	
	 	\addplot coordinates {
		(1,  1)
 		(2, 2.091)
 		(4, 3.5685)
		(8, 3.9135)
	};	
	
	\legend{$N=128$,$N=256$,$N=512$,$N=1024$}
	\endaxis
\endtikzpicture
\caption{ Speedup  plot for the cases $N=128,256,512,1024$, where bandwidth of matrix $\beta_{v}=2$}
\label{fig:speed}
\end{figure}
\begin{figure}[!htbp]
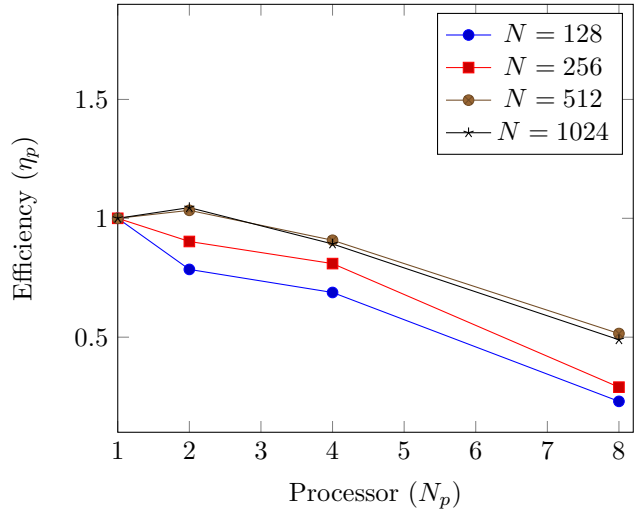

\centering
\tikzpicture
	\axis[
		xlabel=Processor ($N_{p}$),
		ylabel=Efficiency ($\eta_{p}$),
		xmin=1, xmax = 8.2,
		ymin=0.1, ymax = 1.9,
		]
	\addplot coordinates {
		(1, 1)
		(2, 0.785)
		(4, 0.688)
		(8, 0.23)
	};
 	\addplot coordinates {
		(1,  1)
 		(2, 0.9025)
 		(4, 0.8095)
		(8, 0.2897)
	};
 	\addplot coordinates {
		(1,  1)
 		(2, 1.0335)
 		(4, 0.9077)
		(8, 0.5153)
	};	
    \addplot coordinates {
		(1,  1)
 		(2, 1.045)
 		(4, 0.8921)
		(8, 0.4892)
	};
	
	\legend{$N=128$,$N=256$,$N=512$,$N=1024$}
	\endaxis
\endtikzpicture
\caption{ Efficiency plot for the cases $N=128,256,512,1024$, where bandwidth of matrix $\beta_{v}=2$}
\label{fig:eff}
\end{figure} 

\begin{table}[!htbp]
\caption{PAQIF result from 4 processors and for matrix order $64 \times 64 $
bandwidth $\beta=2$.}
\begin{center}
\begin{tabular}{cccc}
\hline
MYID & CPU-time in Second) & CPU-time in Hours) \\
\hline
$1$       & $ 3.0376911163330078 \times 10^{-3}$ & $8.4380308787027993 \times 10^{-7}$\\ 
$2$       & $3.0376911163330078 \times 10^{-3} $ & $8.4380308787027993 \times 10^{-7}$ \\ 
$3$       & $3.0376911163330078\times 10^{-3}$   & $8.4380308787027993 \times 10^{-7}$ \\ 
$0$       & $3.2024383544921875\times 10^{-3}$   & $8.8956620958116323 \times 10^{-7}$ \\
\hline 
 \end{tabular}
 \end{center}
\end{table}
 \begin{table}[!htbp]
 \caption{PAQIF result from 4 processors and for matrix order $128 \times 128 $
 bandwidth $\beta=2$.}
\begin{center}
\begin{tabular}{cccc}
\hline
MYID & CPU-time in Second) & CPU-time in Hours) \\
\hline
$1$       & $ 9.1495513916015625 \times 10^{-3}$ & $2.5415420532226563 \times 10^{-6}$\\ 
$2$       & $9.1459751129150391 \times 10^{-3} $ & $ 2.5405486424763998 \times 10^{-6}$ \\ 
$3$       & $9.1507434844970703\times 10^{-3}$   & $2.5418731901380750 \times 10^{-6}$ \\ 
$0$       & $9.3822479248046875\times 10^{-3}$   & $2.6061799791124132 \times 10^{-6}$ \\
\hline 
 \end{tabular}
 \end{center}
\end{table}
\begin{table}[!htbp]
 \caption{PAQIF result from 4 processors and for matrix order $256 \times 256 $
 bandwidth $\beta=2$.}
\begin{center}
\begin{tabular}{cccc}
\hline
MYID & CPU-time in Sec.) & CPU-time in Hours) \\
\hline
$1$       & $2.2377967834472656\times 10^{-3}$ & $6.2161021762424042 \times 10^{-7}$\\ 
$2$       & $2.2346973419189453\times 10^{-3}$ & $  6.2074926164415153 \times 10^{-7}$ \\ 
$3$       & $1.9750595092773438\times 10^{-3}$   & $5.4862764146592882 \times 10^{-7}$ \\ 
$0$       & $2.5713443756103516\times 10^{-3}$   & $7.1426232655843094 \times 10^{-7}$ \\
\hline 
 \end{tabular}
 \end{center}
\end{table}
\begin{table}[!htbp]
 \caption{PAQIF result from 8 processors and for matrix order $256 \times 256 $
 bandwidth $\beta=2$.}
\begin{center}
\begin{tabular}{cccc}
\hline
MYID & CPU-time in Sec.) & CPU-time in Hours) \\
\hline
 $1$       & $3.9653778076171875\times 10^{-3}$ & $1.1014938354492188 \times 10^{-6}$\\ 
$2$       & $3.9479732513427734\times 10^{-3} $ & $  1.0966592364841036 \times 10^{-6}$ \\ 
$3$       & $3.9269924163818359\times10^{-3}$   & $1.0908312267727323 \times 10^{-6}$ \\ 
$0$       & $4.7206878662109375 \times 10^{-3}$ & $1.3113021850585938 \times 10^{-6}$ \\
$4$       & $3.9708614349365234\times 10^{-3}$   & $1.1030170652601455\times 10^{-6}$ \\
$5$       & $3.9658546447753906\times 10^{-3}$   & $1.1016262902153863 \times 10^{-6}$ \\
$6$       & $3.9696693420410156\times 10^{-3}$   & $1.1026859283447266 \times 10^{-6}$ \\
$7$       & $3.9660930633544922\times 10^{-3}$   & $1.1016925175984701 \times 10^{-6}$ \\
\hline 
 \end{tabular}
 \end{center}
\end{table}
\section{Applications of PAQIF algorithm}
The PAQIF algorithm is an important solver for solving wider class of problems if we do a careful treatment to the general linear and nonlinear elliptic as well as parabolic type PDEs. Howerver, in this article, we restrict our attention in solving the problems related to variational inequality and its application to free boundary problems (in particular, in solving EHL problems).
\subsection{Linear study for convection-diffusion problem}\label{sec:three}
Our specific interest in this subsection is to develop an robust splitting for our EHL model. we consider well known convection-diffusion problem in 1-d and 2-d case as
\begin{example}\label{ex:zero}
 \begin{align}
 \label{eqn39}
   L u = (a(x) u)_{x}-\epsilon u_{xx} = f(x) \quad \forall x \in \Omega \nonumber \\
  u(x) = g(x) \quad \forall x \in \partial \Omega,
  \end{align}
 \end{example}
 and
\begin{example}\label{ex:one}
 \begin{align}
 \label{eqn39}
   L u = (a(x,y) u)_{x}+(b(x,y) u)_{y}-\epsilon \Delta u = f(x,y) \quad \forall (x,y) \in \Omega \nonumber \\
  u(x,y) = g(x,y) \quad \forall (x,y) \in \partial \Omega,
  \end{align}
 \end{example}
where $0 < \epsilon < < 1$.
Then discretization of convective term for $(a u)_{x}$ is performed as
 \begin{align*}
  (a u)_{x}=\frac{a}{h}(u_{i}-u_{i-1})=: L_{1}^{1} 
 \end{align*}
 \begin{align}
 \label{eqn40}
  (a u)_{x}=\frac{a}{h}(u_{i,j}-u_{i-1,j})=: L_{1}^{2} 
 \end{align}
However, this scheme is only $O(h)$ accurate. Our interest here to increase accuracy at least smooth part without contaminating
any wiggle in solution. Consider the Van Leer's $\kappa$-schemes \cite{vanleer} for discretization
term  $(a u)_{x}$ (for $a = \text{const} > 0$) as
\begin{gather*}
  (a u)_{x}=\frac{a}{h}[(u_{i}-u_{i-1})-\frac{\kappa}{2}(u_{i}-u_{i-1})+\frac{1-\kappa}{4}(u_{i}-u_{i-1})
 +\frac{1+\kappa}{4}(u_{i+1}-
 u_{i})-\frac{1-\kappa}{4}(u_{i}-u_{i-2})] \\
  =L_{1}^{1} + L_{\alpha}^{1}+L_{\beta}^{1}+L_{\gamma}^{1}+L_{\delta}^{1}
\end{gather*}
\begin{gather}
\label{eqn41}
  (a u)_{x}=\frac{a}{h}[(u_{i,j}-u_{i-1,j})-\frac{\kappa}{2}(u_{i,j}-u_{i-1,j})+\frac{1-\kappa}{4}(u_{i,j}-u_{i-1,j})\nonumber\\
 +\frac{1+\kappa}{4}(u_{i+1,j}-
 u_{i,j})-\frac{1-\kappa}{4}(u_{i,j}-u_{i-2,j})]\nonumber\\
  =L_{1}^{2} + L_{\alpha}^{2}+L_{\beta}^{2}+L_{\gamma}^{2}+L_{\delta}^{2}
\end{gather}
(similar scheme can be constructed for $a < 0$).
The resulting discrete model Example.~\ref{ex:one} by $\kappa$-scheme (take $\kappa = 0$ here) is denoted by
\begin{align}
\label{eqn42}
[L_{\kappa=0}^{1}]=\frac{a}{h}\left[\begin{matrix}
                      0.25 \ &  -1.25  &  0.75  & 0.25 & 0 \\
                     \end{matrix}
\right]
+
\frac{\epsilon}{h^2}
\left[\begin{matrix}
 -1 \ & \ 2 \ &\  -1 \\
    \end{matrix} 
\right]
\end{align}
\begin{align}
\label{eqn42}
[L_{\kappa=0}^{2}]=\frac{a}{h}\left[\begin{matrix}
                      0.25 \ &  -1.25  & 0.75  & 0.25 & 0 \\
                     \end{matrix}
\right]+\frac{b}{h}
\left[\begin{matrix}
0  \\
0.25 \\
0.75  \\
-1.25 \\
0.25 \\
    \end{matrix} 
\right]
+
\frac{\epsilon}{h^2}
\left[\begin{matrix}
 0 \ & \ -1 \ &\  0 \\
 -1 \ & \ 4 \ &\ -1 \\
   0 \ &\ -1 \ &\ 0 \\
    \end{matrix} 
\right]
\end{align}
In general, above discrete equation.~\ref{eqn39} do not produces $M$-matrix and many iterative splitting on $L_{\kappa}$ diverge.
Therefore, this problem is solved using TVD scheme with help of appropriate flux limiters to prevent a solution from unwanted oscillation.
Now consider $\kappa=-1$ then the second-order upwind scheme looks like ($a > 0$) for one dimensional case
\begin{gather*}
 (au)_{x}=
 \frac{a}{h}[(u_{i}-u_{i-1})+\frac{1}{2}(u_{i}-u_{i-1})
 +\frac{1}{2}(u_{i}-u_{i-1})-\frac{1}{2}(u_{i-1}-u_{i-2})]
 =L_{1}^{1}+L_{\alpha}^{1}+L_{\gamma}^{1}+L_{\delta}^{1}.
\end{gather*}
and for two-dimensional case
\begin{gather}
 (au)_{x}=
 \frac{a}{h}[(u_{i,j}-u_{i-1,j})+\frac{1}{2}(u_{i,j}-u_{i-1,j})
 +\frac{1}{2}(u_{i,j}-u_{i-1,j})-\frac{1}{2}(u_{i-1,j}-u_{i-2,j})]\nonumber\\ \label{eqnn43}
 =L_{1}^{2}+L_{\alpha}^{2}+L_{\gamma}^{2}+L_{\delta}^{2}.
\end{gather}
We enforce Eqn.~\ref{eqnn43} to satisfy TVD condition by multiply limiter functions in the additional terms $L_{\alpha}, L_{\gamma} $ and $L_{\delta}$.
Then following two type of discretization for convection term are presented here as
\begin{gather*}
 (au)_{x}=\frac{a}{h}[(u_{i}-u_{i-1})+\frac{1}{2}\phi(r_{i-1/2})(u_{i}-u_{i-1})
 -\frac{1}{2}\phi(r_{i-3/2})(u_{i-1}-u_{i-2})]
 =L_{1}^{1}+L_{\alpha}^{1}+L_{\gamma}^{1}
\end{gather*}
and
\begin{gather*}
 (au)_{x}=\frac{a}{h}[(u_{i}-u_{i-1})+\frac{1}{2}\phi(r_{i-1/2})(u_{i}-u_{i-1})
 +\frac{1}{2}\phi(r_{i-3/2})(u_{i}-u_{i-1})
 -\frac{1}{2}\phi(r_{i-3/2})(u_{i-1}-u_{i-2})] \\
 =L_{1}^{1}+L_{\alpha}^{1}+L_{\beta}^{1}+L_{\gamma}^{1},
\end{gather*}
where $r_{i-1/2}=\dfrac{(u_{i+1}-u_{i})}{(u_{i}-u_{i-1})}$ and 
$r_{i-3/2}=\dfrac{(u_{i}-u_{i-1})}{(u_{i-1}-u_{i-2})}$.\\
\begin{gather}
 (au)_{x}=\frac{a}{h}[(u_{i,j}-u_{i-1,j})+\frac{1}{2}\phi(r_{i-1/2})(u_{i,j}-u_{i-1,j})\nonumber 
 -\frac{1}{2}\phi(r_{i-3/2})(u_{i-1,j}-u_{i-2,j})]\\ \label{eqn44}
 =L_{1}^{2}+L_{\alpha}^{2}+L_{\gamma}^{2}
\end{gather}
and
\begin{gather}
 (au)_{x}=\nonumber \\
 \frac{a}{h}[(u_{i,j}-u_{i-1,j})+\frac{1}{2}\phi(r_{i-1/2})(u_{i,j}-u_{i-1,j})
 +\frac{1}{2}\phi(r_{i-3/2})(u_{i,j}-u_{i-1,j})
 -\frac{1}{2}\phi(r_{i-3/2})(u_{i-1,j}-u_{i-2,j})]\nonumber \\\label{eqn45}
 =L_{1}^{2}+L_{\alpha}^{2}+L_{\beta}^{2}+L_{\gamma}^{2},
\end{gather}
where $r_{i-1/2}=\dfrac{(u_{i+1,j}-u_{i,j})}{(u_{i,j}-u_{i-1,j})}$ and 
$r_{i-3/2}=\dfrac{(u_{i,j}-u_{i-1,j})}{(u_{i-1,j}-u_{i-2,j})}$.\\
In \cite{Oosterlee,peeyush} represents graph of limiter function $(r,\phi(r))$ on which the resulting convection discretization term
defined in Eqn.~\ref{eqnn43} and Eqn.~\ref{eqn44} enforce to be TVD and higher order accurate.
The discrete representation of Example~\ref{ex:one} using Van-Leer
$\kappa$-scheme in 1-d and 2-d case are defined as 
\begin{align}
\label{eqn46}
 L_{\kappa}^{1}u := \sum_{l_{x} \in \mathcal{I}}
 \mathcal{C}^{(\kappa)}_{l_{x}}u_{i+l_{x}}.
\end{align}
and
\begin{align}
\label{eqn46}
 L_{\kappa}^{2}u := \sum_{l_{x} \in \mathcal{I}}\sum_{l_{y}
 \in \mathcal{I}}\mathcal{C}^{(\kappa)}_{l_{x}l_{y}}u_{i+l_{x},j+l_{y}}.
\end{align}
Moreover, in stencil notation these are represented as
\begin{align}
\label{eqn47}
L_{\kappa}^{1} := 
 \begin{pmatrix}
\mathcal{C}_{-20}^{\kappa} & \mathcal{C}_{-10}^{\kappa} & \mathcal{C}_{00}^{\kappa} & \mathcal{C}_{10}^{\kappa} & \mathcal{C}_{20}^{\kappa}
 \end{pmatrix}
 \end{align} 
 and
\begin{align}
\label{eqn47}
L_{\kappa}^{2} := 
 \begin{pmatrix}
        &         & \mathcal{C}_{02}^{\kappa} &        &        \\
        &         & \mathcal{C}_{01}^{\kappa} &         &         \\
\mathcal{C}_{-20}^{\kappa} & \mathcal{C}_{-10}^{\kappa} & \mathcal{C}_{00}^{\kappa} & \mathcal{C}_{10}^{\kappa} & \mathcal{C}_{20}^{\kappa} \\
        &         & \mathcal{C}_{0-1}^{\kappa} &        &         \\
        &         & \mathcal{C}_{0-2}^{\kappa} &        &              
 \end{pmatrix}
.
\end{align} 
Then the discrete matrix equation $L_{\kappa}^{i}u =f^{i},i=1,2$ are solved efficiently by the use of AQIF method. The related splitting is constructed by 
taking the matrix operator defined in Eqn.~\ref{eqn47}. In particular case, the splitting in $x$-direction is scanned as forward 
(or backward direction depending on flow direction) lexicographical order and it is represented as
$S_{\kappa}=S_{\kappa}^{x_{f}}$ (or $S_{\kappa}^{x_{b}}$). For matrix operator $L_{\kappa}$, the forward splitting
$S_{\kappa}^{x_{f}}$ is defined as 
\begin{gather*}
 L_{\kappa}=L^{x}_{\kappa/2}-(L^{x}_{\kappa/2}-L_{\kappa})=:L^{+}_{\kappa}+L^{0}_{\kappa}+L^{-}_{\kappa},
\end{gather*}
where
\begin{gather*}
{\scriptstyle 
L^{x}_{\kappa/2}:=L^{+}_{\kappa}+L^{0}_{\kappa}=}
{\scriptstyle \begin{pmatrix}
        &         & 0 &        &        \\
        &         & 0 &         &         \\
      0 &       0 & 0 & 0       & 0        \\
        &         & \mathcal{C}_{0-1}^{\kappa} &        &         \\
        &         & \mathcal{C}_{0-2}^{\kappa} &        &              
 \end{pmatrix}}
 + 
 {\scriptstyle \begin{pmatrix}
        &         & 0 &        &        \\
        &         & 0 &         &         \\
0       & \mathcal{C}_{-10}^{\kappa/2} & \mathcal{C}_{00}^{\kappa/2} & \mathcal{C}_{10}^{\kappa/2} & 0 \\
        &         & 0 &        &         \\
        &         & 0 &        &              
 \end{pmatrix} 
 }
\end{gather*}
and therefore overall splitting is
\begin{gather*}
L^{x}_{\kappa/2}u^{n+1}= (L^{x}_{\kappa/2}-L_{\kappa})u^{n}+f.
\end{gather*}
Now for a fixed $x$-line ($m$-grid points in $X$-direction) $$(i,j_{0})_{( 1\le i \le m)}$$, we have the following
\begin{gather*}
 L^{0}_{\kappa}u^{*}=f+L^{0}_{\kappa}u^{n}-(L^{-}_{\kappa}+L^{0}_{\kappa})u^{n}-L^{+}_{\kappa}u^{n+1}.
\end{gather*}
$L^{0}_{\kappa}$ corresponds the operator to the unknowns $u^{*}$ which are scanned simultaneously. $L^{-}_{\kappa}$ corresponds
the operator to the old approximation $u^{n}$, and $L^{+}_{\kappa}$ operator having updated values of $u^{n+1}$.
Now by applying under-relaxation constant $\omega$ in above equation we have
\begin{align*}
u^{n+1}=u^{*}\omega+u^{n}(1-\omega),
\end{align*}
therfore, splitting equation can be rewritten in corresponding change, $\sigma^{n+1}=u^{n+1}-u^{n}$ form as
\begin{align*}
 L^{0}_{\kappa}\sigma^{n+1}=f-(L^{-}_{\kappa}+L^{0}_{\kappa})u^{n}-L^{+}_{\kappa}u^{n+1},\\
 u^{n+1}=u^{n}+\sigma^{n+1}\omega
\end{align*}
Now we construct series of splitting for solving Eqn.~\ref{eqn39} as below.\\
\underline{\bf Splitting : $L_{s0}$}
This splitting is constructed by taking upwind operator $L_{1}$ plus a ``positive" part of the second-order 
operators $L_{\alpha}$ and $L_{\beta}$ from Eqn.~\ref{eqn45} and part of diffusion operator from Eqn.~\ref{eqn47}.
\begin{gather}
\label{eqn48}
L_{\kappa}^{0}u=-\Big\{\frac{\epsilon}{h^2}+\frac{a}{4h}(5-3\kappa)\Big\}u_{i-1,j}+\Big\{\frac{a}{h}\Big(\frac{2-\kappa}{2}
+\frac{1-\kappa}{4}\Big)
+\frac{4\epsilon}{h^2}\Big\}u_{i,j} + \Big\{-\frac{\epsilon}{h^{2}}\Big\}u_{i+1,j}\nonumber\\
L_{\kappa}^{+}u=\Big\{-\frac{\epsilon}{h^{2}}\Big\}u_{i,j-1} \nonumber\\
L_{\kappa}^{-}u=\Big\{\frac{a}{h}\Big(\frac{1-\kappa}{4}\Big)\Big\}u_{i-2,j}+\Big\{\frac{a}{h}\Big(\frac{1-\kappa}{4}\Big)
 \Big\}u_{i-1,j}+ \nonumber \\
\Big\{-\frac{a}{h}\Big(\frac{1+\kappa}{4}\Big)\Big\}u_{i,j}
 +\Big\{\frac{a}{h}\Big(\frac{1+\kappa}{4}\Big)\Big\}u_{i+1,j}+\Big\{-\frac{\epsilon}{h^{2}}\Big\}u_{i,j+1}.
\end{gather}
\underline{\bf Splitting : $Ls1$}
This splitting is constructed taking upwind operator $L_{1}$ plus a ``positive" part of the second-order 
operators $L_{\alpha}$ from Eqn.~\ref{eqn44} and part of diffusion operator from Eqn.~\ref{eqn47}.
\begin{gather}
\label{eqn49}
L_{\kappa}^{0}u=\Big\{-\frac{a}{h}\Big(\frac{2-\kappa}{2}\Big)-\frac{\epsilon}{h^2}\Big\}u_{i-1,j}+\Big\{\frac{a}{h}\Big(\frac{2-\kappa}{2}\Big)
 +\frac{4\epsilon}{h^2}\Big\}u_{i,j}+\Big\{-\frac{\epsilon}{h^{2}}\Big\}u_{i+1,j} \nonumber \\
L_{\kappa}^{+}u=\Big\{-\frac{\epsilon}{h^{2}}\Big\}u_{i,j-1}\nonumber\\
L_{\kappa}^{-}u=\Big\{\frac{a}{h}\Big(\frac{1-\kappa}{4}\Big)\Big\}u_{i-2,j}+\Big\{\frac{a}{h}\Big(\frac{1-\kappa}{4}\Big)
 \Big\}u_{i-1,j}+\Big\{-\frac{a}{h}\Big(\frac{1+\kappa}{4}\Big)\Big\}u_{i,j} \nonumber\\
  +\Big\{\frac{a}{h}\Big(\frac{1+\kappa}{4}\Big)\Big\}u_{i+1,j}+\Big\{-\frac{\epsilon}{h^{2}}\Big\}u_{i,j+1}
\end{gather}
\underline{\bf Splitting : ${Ls2}$}
In this case splitting coefficients $\mathcal{C}_{**}^{\kappa}$ correspond only to the first-order upwind operator $L_{1}$ of a 
discretized Eqn.~\ref{eqn44} plus diffusion operator.
\begin{gather}
\label{eqn50}
{L_{\kappa}^{0}u=\Big\{-\frac{a}{h}-\frac{\epsilon}{h^{2}}\Big\}u_{i-1,j}+\Big\{\frac{a}{h}
 +\frac{4\epsilon}{h^2}\Big\}u_{i,j}+\Big\{-\frac{\epsilon}{h^{2}}\Big\}u_{i+1,j} }\nonumber \\
 {L_{\kappa}^{+}u=\Big\{-\frac{\epsilon}{h^{2}}\Big\}u_{i,j-1}}\nonumber\\
 { L_{\kappa}^{-}u=\Big\{\frac{a}{h}\Big(\frac{1-\kappa}{4}\Big)\Big\}u_{i-2,j}+\Big\{-\frac{a}{h}\Big(\frac{1-3\kappa}{4}\Big)
 \Big\}u_{i-1,j}+\Big\{-\frac{a}{h}\Big(\frac{1+3\kappa}{4}\Big)\Big\}u_{i,j}} \nonumber\\
 { +\Big\{\frac{a}{h}\Big(\frac{1+\kappa}{4}\Big)\Big\}u_{i+1,j}+\Big\{-\frac{\epsilon}{h^{2}}\Big\}u_{i,j+1}}
\end{gather}
\underline{\bf Splitting : $Ls3$}
The third splitting named as $\kappa$- distributive line relaxation is constructed by assuming a ghost variable $\sigma_{*}$ (with the 
same cardinality as $\sigma$) such that $\sigma = \mathcal{D}\sigma_{*}$, where matrix $\mathcal{D}$ comes due to distributive change of the relaxation in other way we construct line-wise distributive splitting as
\begin{gather}
\label{eqn51}
 u_{i,j}^{n+1}=u^{n}_{i,j}+\sigma_{i,j}-\frac{(\sigma_{i+1,j}+\sigma_{i-1,j}+\sigma_{i,j+1}+\sigma_{i,j-1})}{4}
\end{gather}
This splitting is understood in the following way:
First, discretize Example~\ref{ex:one} by $\kappa$-scheme and get the equation of the form as
\begin{gather*}
L^{x}_{\kappa/2}u^{n+1}= f', \quad \text{where }  f'=(L^{x}_{\kappa/2}-L_{\kappa})u^{n}+f.
\end{gather*}
Now in the above splitting equation put the value of $u^{n+1}$ from Eqn.~\ref{eqn51} and apply distributive splitting in the form of right preconditioner defined below.
\begin{gather*}
L^{x}_{\kappa/2}\sigma^{n+1}= R^{n}\quad \text{and } L^{x}_{\kappa/2}\mathcal{D}\sigma^{n+1}_{*}= R^{n},
\end{gather*}
where the updated change in pressure and residual equation are denoted as
\begin{align*}
 \sigma^{n+1}=\mathcal{D}\sigma^{n+1}_{*} \text{ and } R^{n}=L^{x}_{\kappa/2}u^{n+1}- f'
\end{align*}
respectively.
In other way, line distributive splitting consists of following two steps;
In first step it calculates new ghost value approximation change $\sigma^{n+1}_{*}$. Second step calculates new approximation
change $\sigma^{n+1}$.\\
Now applying above splitting along the $x$-direction in Example~\ref{ex:one}, the diffusive term is computed as
\begin{gather}
\label{eqn52}
{ -\epsilon \Big[\Big\{ u_{i+1,j}+\sigma_{i+1}-\frac{(\sigma_{i} + \sigma_{i+2})}{4} \Big\}
 -\Big\{ u_{i,j}+\sigma_{i}-\frac{(\sigma_{i-1} + \sigma_{i+1})}{4} \Big\}\Big]\Big/h^{2} }\nonumber\\
{  -\epsilon \Big[\Big\{ u_{i-1,j}+\sigma_{i-1}-\frac{(\sigma_{i-2} + \sigma_{i})}{4} \Big\}
 -\Big\{ u_{i,j}+\sigma_{i}-\frac{(\sigma_{i-1} + \sigma_{i+1})}{4} \Big\}\Big]\Big/h^{2} }\nonumber\\
{  -\epsilon \Big[\Big\{ u_{i,j+1}-\frac{\sigma_{i}}{4} \Big\}
 -\Big\{ u_{i,j}+\sigma_{i}-\frac{(\sigma_{i-1} + \sigma_{i+1})}{4} \Big\}\Big]\Big/h^{2} }\nonumber\\
{  -\epsilon \Big[\Big\{ u_{i,j-1}-\frac{\sigma_{i}}{4} \Big\}
 -\Big\{ u_{i,j}+\sigma_{i}-\frac{(\sigma_{i-1} + \sigma_{i+1})}{4} \Big\}\Big]\Big/h^{2}}.
\end{gather}
and convection term is computed as
\begin{gather}
\label{eqn53}
+\Big[\frac{a_{i+1/2,j}(2+\kappa)}{2h}\Big\{ u_{i,j}+\sigma_{i}-\frac{(\sigma_{i-1} + \sigma_{i+1})}{4} \Big\}\nonumber \\
 -\frac{a_{i-1/2,j}(2+\kappa)}{2h}\Big\{ u_{i-1,j}+\sigma_{i-1}-\frac{(\sigma_{i-2} + \sigma_{i})}{4} \Big\}\Big]
\end{gather}
Other part of convective term which comes from Van-leer discretization do not contain any distributive term as above explained and kept in right hand side during relaxation and overall splitting is written as follows
\begin{gather}
{ \Big(\frac{\epsilon}{4h^{2}}+\frac{a_{i-1/2,j}(2+\kappa)}{8h}\Big)\sigma_{i-2}
-\Big(\frac{7\epsilon}{4h^{2}}+\frac{a_{i+1/2,j}(2+\kappa)}{2h}+
\frac{a_{i-1/2,j}(2+\kappa)}{8h}\Big)\sigma_{i-1} }\nonumber \\
{ +\Big(\frac{20\epsilon}{4h^{2}}+\frac{a_{i+1/2,j}(2+\kappa)}{2h}
+\frac{a_{i-1/2,j}(2+\kappa)}{8h}\Big)\sigma_{i} }\nonumber \\ 
{ -\Big(\frac{8\epsilon}{4h^{2}}+\frac{a_{i+1/2,j}(2+\kappa)}{2h}\Big)\sigma_{i+1}
+\frac{\epsilon}{4h^{2}}\sigma_{i+2} } \nonumber \\ \label{eqn54}
{ =R_{i,j}+\Big\{\frac{1+\kappa}{4}(u_{i+1,j}-u_{i,j})-\frac{1-\kappa}{4}(u_{i-1,j}-u_{i-2,j})\Big\}\Big]}
\end{gather}
after solving above equation for $\sigma$ along $x$ line direction updated solution $u^{n+1}$ is evaluated as 
\begin{gather*}
 u_{i,j}^{n+1}=u^{n}_{i,j}+\sigma_{i,j}-\frac{(\sigma_{i+1,j}+\sigma_{i-1,j}+\sigma_{i,j+1}+\sigma_{i,j-1})}{4}.
\end{gather*}
However, above splitting $Ls3$ Eqn.~\ref{eqn54} is not robust and very rarely use in practice.\\
\subsection{Application in solving Variational inequality and LCP}
In the last section, we have shown a series of splittings for solving 
convection-diffusion type problems. This idea can be generalizing for more
general variational inequality and LCP related applications. In this section, we will discuss convergence criterion for solving PAQIF algorithm for general variational inequality and LCP problems.\\
Let us consider domain $\Omega \in \mathbb{R}^2$ with boundary $\partial \Omega$, and consider known functions $f$ and $g$. 
Then find $u$ in a weak sense such that these inequalities hold
\begin{example}
  \begin{align*}
  -(a(x,y)u)_{x}+\epsilon \Delta u \le f(x,y) \quad \forall x,y \in \Omega \nonumber\\
   u(x,y) \ge 0 \quad \forall x,y \in \Omega,\\
   u(x,y)[(a(x,y)u)_{x}-\epsilon \Delta u - f(x,y)]=0  \quad \forall x,y \in \Omega,\\
  u(x,y) = g(x,y) \quad \forall x,y \in \partial \Omega.
 \end{align*}
\end{example}

\begin{example}
 \begin{align*}
  -(a(x,y)u)_{x}+\epsilon \Delta u \le f(x,y) \quad \forall x,y \in \Omega \nonumber\\
   u(x,y) \ge 0 \quad \forall x,y \in \Omega,\\
   u(x,y)[(a(x,y)u)_{x}-\epsilon \Delta u - f(x,y)]=0  \quad \forall x,y \in \Omega,\\
  u(x,y) = g(x,y) \quad \forall x,y \in \partial \Omega.
 \end{align*}
\end{example} 
Therefore, discrete version of above problem (finite difference or finite volume) is written in the matrix form
 \begin{align}
 \label{eqnnew1}
  Lu \le f, \nonumber \\
   u \ge 0,\nonumber \\
   u[L u - f]=0,
  \end{align}
where $L$ is a $M$-matrix of order $m\times m$, $u$ and $f$ are $m\times 1$-column vector.   
It is well known that solving above discrete problem is equivalent to solving quadratic
minimization problem of the form
\begin{align} 
\label{eqnnew2}
 G(u)=\frac{1}{2}u^{T}Lu-f^{T}u, \nonumber \\
 \min_{u \in \mathbb{R}^m\times1} G(u),
\end{align}
subjected to the constraints
\begin{align*}u\ge0.\end{align*}
\begin{theorem}\label{thrm:lcp}
 Let $u^{n}$ and $f^{n}$ are $m\times 1$-column vectors achieved by splitting algorithm (*),
 \begin{align*}
 L^{0}_{\kappa}\sigma^{n+1}=f-(L^{-}_{\kappa}+L^{0}_{\kappa})u^{n}-L^{+}_{\kappa}u^{n+1},\\
 \sigma^{n+1}=\max\{0,\sigma^{n+1}\},\\
 u^{n+1}=u^{n}+\sigma^{n+1}\omega,
\end{align*}
where $0 < \omega < 1$ then we have $u^{n} \rightarrow u$ and 
$f^{n} \rightarrow f$ such that $u$ and $f$ is a solution of LCP problem.
\end{theorem}
\begin{proof}
For the proof of this theorem we refer to see Cryer \cite{Cryer}.
\end{proof}
The following error estimates are easily established for LCP problem for algorithm described above.
\begin{lemma}\label{lemma:lcp}
 Let $u$ is the exact solution of LCP problem define in Eqn.~\ref{eqnnew1}, also let $u^{n+1}$ is approximate solution
 obtained by the splitting of the form
  \begin{align*}
 L^{0}_{\kappa}\sigma^{n+1}=f-(L^{-}_{\kappa}+L^{0}_{\kappa})u^{n}-L^{+}_{\kappa}u^{n+1},\\
 \sigma^{n+1}=\max\{0,\sigma^{n+1}\},\\
 u^{n+1}=u^{n}+\sigma^{n+1}\omega
\end{align*}
Then the following conditions hold
\begin{align*}
 \|u-u^{n+1}\|_{2} \le C_{2}\|u^{n+1}-u^{n}\|_{2} \\
  \|u-u^{n+1}\|_{1} \le C_{1}\|u^{n+1}-u^{n}\|_{1} \\
  \|u-u^{n+1}\|_{\infty} \le C_{\infty}\|u^{n+1}-u^{n}\|_{\infty}.
\end{align*}
\end{lemma}
\begin{proof}
Proof is followed from Lemma 2.2 mentioned in \cite{free_boundary}.
\end{proof}
\subsection{Application in solving steady state EHL problem}
Now, we illustrate splitting for compressible EHL model in the form of inequalities as
\begin{example}[Point contact case]\label{ex:three}
 \begin{align}
 \label{eqn55}
 (a(x,y)\mathcal{H}(u))_{x}-\epsilon \Delta u \ge f(x,y) \quad \forall x,y \in \Omega \nonumber\\
 u(x,y) \ge 0 \quad \forall x,y \in \Omega,\nonumber\\
 u(x,y)[(a(x,y)\mathcal{H}(u))_{x}-\epsilon \Delta u - f(x,y)]=0 \quad \forall x,y \in \Omega,\nonumber\\
 u(x,y) = g(x,y) \quad \forall x,y \in \partial \Omega,\nonumber\\
\mathcal{H}(u)=\mathcal{H}_{0}+\frac{x^{2}+y^{2}}{2} + 
\frac{2}{\pi^{2}}\int_{-\infty}^{\infty} \int_{-\infty}^{\infty}\frac{u(x^{'},y^{'})dx^{'}dy^{'}}{\sqrt{(x-x^{'})^2+(y-y^{'})^2}}.
  \end{align}
The dimensionless force balance equation are defined as follows
\[ \int_{-\infty}^{\infty} \int_{-\infty}^{\infty}u(x',y') dx'dy' = \frac{3\pi}{2}, \textit{ for point contact }\]
 \end{example}
 \begin{example}[Line contact case]\label{ex:four}
 \begin{align}
 \label{eqn55}
 (a(x)\mathcal{H}(u))_{x}-\epsilon u_{xx} \ge f(x) \quad \forall x \in \Omega \nonumber\\
 u(x) \ge 0 \quad \forall x \in \Omega,\nonumber\\
 u(x)[(a(x)\mathcal{H}(u))_{x}-\epsilon \Delta u - f(x)]=0 \quad \forall x \in \Omega,\nonumber\\
 u(x) = g(x) \quad \forall x \in \partial \Omega,\nonumber\\
\mathcal{H}(u)=\mathcal{H}_{0}+\frac{x^{2}}{2} - 
\frac{1}{\pi}\int\limits_{-\infty}^{\infty} \log{|(x-x^{'}|)}u(x^{'}) dx^{'}.
  \end{align}
The dimensionless force balance equation are defined as follows
 \[\int_{-\infty}^{\infty} u(x') dx' = \frac{\pi}{2}, \textit{ for line contact }\]
 Here term $\epsilon$  is defined as
\begin{equation*}
 \epsilon = \frac{\rho \mathcal{H}^{3}}{\eta\lambda},
\end{equation*}
where $\rho$ is dimensionless density of lubrication, $\eta$ is dimensionless viscosity of lubrication and
speed parameter
\begin{align*}
   \lambda= \dfrac{6\eta_{0}u_{s}R^{2}}{a^{3}p_{H}}.
  \end{align*}
The non-dimensionless viscosity $\eta$ is defined according to 
\begin{align*}
 \eta(u) = \exp\Bigg\{ \Bigg( \dfrac{\alpha p_{0}}{z}  \Bigg) 
 \Bigg(-1+\Big(1+\dfrac{{u}p_{H}}{p_{0}}\Big)^{z}   \Bigg)   \Bigg\}.
\end{align*}
Dimensionless density $\rho$ is given by
\begin{align*}
 \rho(u) = \dfrac{0.59 \times 10^{9} + 1.34 u p_{H}}{0.59 \times 10^{9} + u p_{H}}.
\end{align*}
For incompressible EHL. we take $\rho=$ and $\eta=1$.
(For the parameters details of above Example 5 and 6 see appendix-\ref{app:one}).
 \end{example}
 \subsubsection{Film thickness Calculation}
{\bf Case 1: Line Contact}
Let us define deformation integral $\mathcal{D}_{f}$ as 
\begin{align}
\label{eqn101}
  \mathcal{D}_{f}(x) = \frac{1}{\pi}\int\limits_{-\infty}^{\infty} \log{|(x-x^{'}|)}u(x^{'}) dx^{'}.
\end{align}
We approximate the above integral Eqn.~\ref{eqn14} taking pressure $u$ as piecewise constant function namely $u^{h}_{i'}$ on sub-domain
\begin{align}
\label{eqn17}
{ \Omega^{h}=\Big\{ (x) \in \mathbb{R}\Big|x_{i^{'}}-\frac{h}{2} \le x \le x_{i^{'}}+\frac{h}{2} \Big\}}.
\end{align} 
and discrete deformation
\begin{align}
\label{eqn16}
 {\mathcal{D}_{f}}_{i} 
 = \mathcal{D}_{f}(x_{i})\approx\frac{1}{\pi}\sum_{i'=0}^{n_{x}}
{\mathcal{G}^{h}}_{i,i^{'}}u^{h}_{i'},
\end{align}
where the coefficients $\mathcal{G}^{h}_{i,i^{'}}$ is written as
\begin{align}
\label{eqn17}
\mathcal{G}^{h}_{i,i^{'}} = \int\limits_{x_{i^{'}}-\frac{h}{2}}^{x_{i^{'}}
+\frac{h}{2}} {\log|(x-x^{'})|}dx^{'}
\end{align}
and evaluated analytically.
Above integration are defined as
\small
\begin{align}
\label{eqn18}
\mathcal{G}^{h}_{i,i^{'}} =\Big\{
 |x_{+}|(\log|x_{+}|-1) -|x_{-}|(\log|x_{-}|-1) \Big\},
\end{align}
\normalsize
where 
\begin{align*}
 x_{+} = x_{i}-x_{i^{'}}+\frac{h}{2},\quad
 x_{-} = x_{i}-x_{i^{'}}-\frac{h}{2} 
\end{align*}
Therefore, film thickness for line contact in discretized form is written as
\begin{equation}
\label{eqn19}
 \mathcal{H}_{i}^{h} := \mathcal{H}_{0}+\frac{x^2_{i}}{2}-
 \sum_{i'}\mathcal{G}^{h}_{|i-i'|}{u}_{i^{'}}^{h}
\end{equation}
{\bf Case 2: Point Contact}
Let us define deformation integral $\mathcal{D}_{f}$ as 
\begin{align}
\label{eqn1016}
  \mathcal{D}_{f}(x,y) = \frac{2}{\pi^2}\int\limits_{-\infty}^{\infty} \int\limits_{-\infty}^{\infty}
 \frac{u(x^{'},y^{'})}{\sqrt{(x-x^{'})^2+(y-y^{'})^2}}dx^{'}dy^{'}.
\end{align}
We approximate the above integral Eqn.~\ref{eqn14} taking pressure $u$ as piecewise constant function namely $u^{h}_{i',j'}$ on sub-domain
\begin{align}
\label{eqn17}
{ \Omega^{h}=\Big\{ (x,y) \in \mathbb{R}^{2}\Big|x_{i^{'}}-\frac{h}{2} \le x \le x_{i^{'}}
+\frac{h}{2},y_{j^{'}}-\frac{h}{2} \le y \le y_{j^{'}}
+\frac{h}{2} \Big\}}.
\end{align} 
and discrete deformation
\begin{align}
\label{eqn16}
 {\mathcal{D}_{f}}_{i,j} 
 = \mathcal{D}_{f}(x_{i},y_{j})\approx\frac{2}{\pi^2}\sum_{i'=0}^{n_{x}}
 \sum_{j'=0}^{n_{y}}{\mathcal{G}^{h}}_{i,i^{'},j,j^{'}}u^{h}_{i',j'},
\end{align}
where the coefficients $\mathcal{G}^{h}_{i,i^{'},j,j^{'}}$ is written as
\begin{align}
\label{eqn17}
\mathcal{G}^{h}_{i,i^{'},j,j^{'}} = \int\limits_{x_{i^{'}}-\frac{h}{2}}^{x_{i^{'}}
+\frac{h}{2}} \int\limits_{y_{j^{'}}-\frac{h}{2}}^{y_{j^{'}}+\frac{h}{2}}
\frac{1}{\sqrt{(x-x^{'})^2+(y-y^{'})^2}}dx^{'}dy^{'}
\end{align}
and evaluated analytically.
Above integration Eqn.~\ref{eqn17} yields nine different results for the cases that are defined as
\[ x_{i} < x_{i^{'}},  x_{i} > x_{i^{'}}, x_{i} = x_{i^{'}} \text{ and }
 y_{j} < y_{j^{'}}, y_{j} > y_{j^{'}}, y_{j} = y_{j^{'}} \]
respectively. The nine results are combined into one expression
\small
\begin{align}
\label{eqn18}
\mathcal{G}^{h}_{i,i^{'},j,j^{'}} =\frac{2}{\pi^{2}}\Big\{
 |x_{+}|\sinh^{-1}(\frac{y_{+}}{x_{+}})+|y_{+}|\sinh^{-1}(\frac{x_{+}}{y_{+}}) 
-|x_{-}|\sinh^{-1}(\frac{y_{+}}{x_{-}}) \nonumber \\ -|y_{+}|\sinh^{-1}(\frac{x_{-}}{y_{+}})   
-|x_{+}|\sinh^{-1}(\frac{y_{-}}{x_{+}})-|y_{-}|\sinh^{-1}(\frac{x_{+}}{y_{-}})\nonumber \\
+|x_{-}|\sinh^{-1}(\frac{y_{-}}{x_{-}})+|y_{-}|\sinh^{-1}(\frac{x_{-}}{y_{-}}) \Big\},
\end{align}
\normalsize
where 
\begin{align*}
 x_{+} = x_{i}-x_{i^{'}}+\frac{h}{2},\quad
 x_{-} = x_{i}-x_{i^{'}}-\frac{h}{2}  \nonumber \\ 
 y_{+} = y_{j}-y_{j^{'}}+\frac{h}{2},\quad 
 y_{-} = y_{j}-y_{j^{'}}-\frac{h}{2}.
\end{align*}
Therefore film thickness in discretized form is written as
\begin{equation}
\label{eqn19}
 \mathcal{H}_{i,j}^{h} := \mathcal{H}_{0}+\frac{x^2_{i}}{2}+\frac{y^{2}_{j}}{2}
 +\sum_{i'} \sum_{j'}\mathcal{G}^{h}_{|i-i'|,|j-j'|}{u}_{i^{'},j^{'}}^{h}
\end{equation}
For incompressible EHL problem $\kappa$-line distributive Jacobi splitting is written as 
consider the convection term of above Example~\ref{ex:three} as
\begin{align}
\label{eqn56}
 \frac{\partial h}{\partial x}=\frac{1}{h_{x}}\Big[(\mathcal{H}_{i,j}-\mathcal{H}_{i-1,j})
 -\frac{\kappa}{2}(\mathcal{H}_{i,j}-\mathcal{H}_{i-1,j})+\nonumber\\
 \frac{1+\kappa}{4}(\mathcal{H}_{i+1,j}-\mathcal{H}_{i,j})-\frac{1-\kappa}{4}(\mathcal{H}_{i-1,j}-\mathcal{H}_{i-2,j})\Big]
\end{align}
Now we will consider the following \underline{\bf Splitting : ${Ls4}$}
\begin{gather}
\label{eqn57}
 { -\epsilon \Big[\Big\{ u_{i+1,j}+\sigma_{i+1}-\frac{(\sigma_{i} + \sigma_{i+2})}{4} \Big\}
 -\Big\{ u_{i,j}+\sigma_{i}-\frac{(\sigma_{i-1} + \sigma_{i+1})}{4} \Big\}\Big]\Big/h^{2}_{x} }\nonumber\\
{ -\epsilon \Big[\Big\{ u_{i-1,j}+\sigma_{i-1}-\frac{(\sigma_{i-2} + \sigma_{i})}{4} \Big\}
 -\Big\{ u_{i,j}+\sigma_{i}-\frac{(\sigma_{i-1} + \sigma_{i+1})}{4} \Big\}\Big]\Big/h^{2}_{x} }\nonumber\\
{ -\epsilon \Big[\Big\{ u_{i,j+1}-\frac{\sigma_{i}}{4} \Big\}
 -\Big\{ u_{i,j}+\sigma_{i}-\frac{(\sigma_{i-1} + \sigma_{i+1})}{4} \Big\}\Big]\Big/h^{2}_{x} }\nonumber\\
{ -\epsilon \Big[\Big\{ u_{i,j-1}-\frac{\sigma_{i}}{4} \Big\}
 -\Big\{ u_{i,j}+\sigma_{i}-\frac{(\sigma_{i-1} + \sigma_{i+1})}{4} \Big\}\Big]\Big/h^{2}_{x} }\nonumber\\
{ -\frac{1}{h_{x}}\Big[\Big(\frac{2-\kappa}{2}\Big)
\Big(\sum_{k=i-1}^{i+1}\sigma \mathcal{G}_{ikjj}\sigma_{k}-\sum_{k=i-2}^{i}\sigma \mathcal{G}_{i-1kjj} \sigma_{k}\Big) }\nonumber\\
{ -\Big\{\frac{1+\kappa}{4}(\mathcal{H}_{i+1,j}-\mathcal{H}_{i,j})-\frac{1-\kappa}{4}(\mathcal{H}_{i-1,j}-\mathcal{H}_{i-2,j})\Big\}\Big]=f_{i,j}}
\end{gather}
Another possibility is to consider the following splitting as 
\begin{align}
\label{eqn58}
{ \frac{\partial h}{\partial x}=\frac{1}{h_{x}}\Big[(\mathcal{H}_{i,j}-\mathcal{H}_{i-1,j})
 -\frac{\kappa}{2}(\mathcal{H}_{i,j}-\mathcal{H}_{i-1,j})+ }\nonumber\\
{ \frac{1+\kappa}{4}(\mathcal{H}_{i+1,j}-\mathcal{H}_{i,j})-\frac{1-\kappa}{4}(\mathcal{H}_{i-1,j}-\mathcal{H}_{i,j}
 +\mathcal{H}_{i,j}-\mathcal{H}_{i-2,j})\Big]}
\end{align}
Hence overall equation is rewritten as \underline{\bf Splitting : ${Ls5}$}
\begin{gather}
\label{eqn59}
{ -\epsilon \Big[\Big\{ u_{i+1,j}+\sigma_{i+1}-\frac{(\sigma_{i} + \sigma_{i+2})}{4} \Big\}
 -\Big\{ u_{i,j}+\sigma_{i}-\frac{(\sigma_{i-1} + \sigma_{i+1})}{4} \Big\}\Big]\Big/h^{2}_{x} }\nonumber\\
{ -\epsilon \Big[\Big\{ u_{i-1,j}+\sigma_{i-1}-\frac{(\sigma_{i-2} + \sigma_{i})}{4} \Big\}
 -\Big\{ u_{i,j}+\sigma_{i}-\frac{(\sigma_{i-1} + \sigma_{i+1})}{4} \Big\}\Big]\Big/h^{2}_{x} }\nonumber\\
{ -\epsilon \Big[\Big\{ u_{i,j+1}-\frac{\sigma_{i}}{4} \Big\}
 -\Big\{ u_{i,j}+\sigma_{i}-\frac{(\sigma_{i-1} + \sigma_{i+1})}{4} \Big\}\Big]\Big/h^{2}_{x} }\nonumber\\
{ -\epsilon \Big[\Big\{ u_{i,j-1}-\frac{\sigma_{i}}{4} \Big\}
 -\Big\{ u_{i,j}+\sigma_{i}-\frac{(\sigma_{i-1} + \sigma_{i+1})}{4} \Big\}\Big]\Big/h^{2}_{x} }\nonumber\\
{ -\frac{1}{h_{x}}\Big[\Big(\frac{2-\kappa}{2}+\frac{1-\kappa}{4}\Big)
\Big(\sum_{k=i-1}^{i+1}\sigma \mathcal{G}_{ikjj}\sigma_{k}-\sum_{k=i-2}^{i}\sigma \mathcal{G}_{i-1kjj} \sigma_{k}\Big) }\nonumber\\
 { -\Big\{\frac{1+\kappa}{4}(\mathcal{H}_{i+1,j}-\mathcal{H}_{i,j})-\frac{1-\kappa}{4}(\mathcal{H}_{i,j}-\mathcal{H}_{i-2,j})\Big\}\Big]=f_{i,j}}.
\end{gather}
More general discussion on convergence of these splittings are given in Section~\ref{sec:five}.
\subsubsection{TVD Implementation in line and point contact model problem}\label{sec:four}
In this section, we implement the splitting discussed in the last section~\ref{sec:three} and allow to extend it in EHL model.
A hybrid splittings are presented here. These splittings are determined by measuring the value 
$\min\Big(\frac{\epsilon(x)}{h_{x}}\Big)$
for one-dimensional EHL line contact case and 
$\min\Big(\frac{\epsilon(x,y)}{h_{x}},\frac{\epsilon(x,y)}{h_{y}}\Big)$
for two-dimensional point contact case.
These values are treated as switching parameter to perform two different splitting together while moving $x$ direction during the iteration. 
If the value of
$$
\begin{cases}
\min\Big(\frac{\epsilon(x)}{h_{x}}\Big) > 0.6, \textit{ for 1-d case }\\
\min\Big(\frac{\epsilon(x,y)}{h_{x}},\frac{\epsilon(x,y)}{h_{y}}\Big) > 0.6, \textit { for 2-d case }
\end{cases}
$$
then we apply $x$- direction line splitting otherwise, $x$- direction weighted change line splitting is incorporated in other words
\begin{align}
\label{eqn60}
L_{hs1}= \begin{cases} 
     L_{s1}\text{-splitting} &\text{ If } \min\Big(\frac{\epsilon(x,y)}{h_{x}},\frac{\epsilon(x,y)}{h_{y}}\Big) > 0.6 \\
     L_{s4}\text{-splitting} &\text{ If } \min\Big(\frac{\epsilon(x,y)}{h_{x}},\frac{\epsilon(x,y)}{h_{y}}\Big) \le 0.6.
   \end{cases}
\end{align}
\begin{align}
\label{eqn61} L_{hs2}= \begin{cases} 
     L_{s0}\text{-splitting} &\text{ If } \min\Big(\frac{\epsilon(x,y)}{h_{x}},\frac{\epsilon(x,y)}{h_{y}}\Big) > 0.6 \\
     L_{s5}\text{-splitting} &\text{ If } \min\Big(\frac{\epsilon(x,y)}{h_{x}},\frac{\epsilon(x,y)}{h_{y}}\Big) \le 0.6.
   \end{cases}
\end{align}
These constructions are well justified as the region where $\epsilon$ tends to zero, we end up having an ill-conditioned matrix system in the form of dense kernel matrix appear in film thickness term.
In next section, we define these two splitting in more general form having limiter function involve in the splitting.
\subsubsection{Limiter based Newton-Raphson method}\label{subsec:num1}
EHL point contact problem is solved in the form of LCP and therefore in this Section we seek an efficient splitting for Reynolds equation
iterate along $x$-line direction as well as $y$-line direction to obtain the pressure solution. Now by using Theorem~\ref{thrm:lcp} and Lemma~\ref{lemma:lcp} we prove the convergence of the EHL solution.
This splitting is explained in the following way: First calculate updated pressure in $x$-line direction
as $\bar{u}_{i,j} = \tilde{u}_{i,j} + \sigma_{i}$ keeping $j$ fix  at a time for all $j$ in $y$-direction and then 
apply change $\sigma_{i}$ immediately to update the pressure $\tilde{u}$. The successive pressure change $\sigma_{i}$ along the $x$-direction can be calculated as below
\begin{align} 
\label{eqn62}
{ \frac{\epsilon^{X}_{i+1/2,j}[({u}_{i+1,j} + \sigma_{i+1})-({u}_{i,j}+ \sigma_{i})]
 +\epsilon^{X}_{i-1/2,j}[({u}_{i-1,j} + \sigma_{i-1})-({u}_{i,j}+ \sigma_{i})]}{h_{x}}} \nonumber \\
{ +  \frac{\epsilon^{Y}_{i,j+1/2}[{u}_{i,j+1}-({u}_{i,j}+ \sigma_{i})]
 +\epsilon^{Y}_{i,j-1/2}[{u}_{i,j-1}-({u}_{i,j}+ \sigma_{i})]}{h_{y}} }  \nonumber \\
{ -h_{y}((\rho \mathcal{H})^{*}_{i+1/2,j}-(\rho \mathcal{H})^{*}_{i-1/2,j}) = 0},
\end{align}
where terms read as
\small
\begin{align}
\label{eqn63}
 \epsilon^{X}_{i \pm 1/2,j}\stackrel{\text{defn}}{:=} h_{y}\epsilon_{i \pm 1/2,j} \nonumber,\quad
 \epsilon^{Y}_{i,j \pm 1/2}\stackrel{\text{defn}}{:=} h_{x}\epsilon_{i,j \pm 1/2} \nonumber, \nonumber\\
 \epsilon_{i \pm 1/2,j}\stackrel{\text{defn}}{:=} (\epsilon_{i,j}+\epsilon_{i\pm 1,j})/2,\quad
 \epsilon_{i,j \pm 1/2}\stackrel{\text{defn}}{:=} (\epsilon_{i,j}+\epsilon_{i,j\pm 1})/2,
\end{align}
where
\begin{gather*}
 \epsilon_{i ,j}=\frac{\rho(i,j) \mathcal{H}^{3}(i,j)}{\eta(i,j)\lambda}.
\end{gather*}
\begin{align}
\label{eqn64}
{ (\rho \mathcal{H})^{*}_{i+1/2,j}\stackrel{\text{def}}{:=}(\check{\rho} \bar{\mathcal{H}})_{i,j}+
\frac{1}{2}\phi(r_{i+1/2})((\check{\rho}\bar{\mathcal{H}})_{i+1,j}-(\check{\rho} \bar{\mathcal{H}})_{i,j})}
\end{align}
\begin{align}
\label{eqn65}
{ (\rho \mathcal{H})^{*}_{i-1/2,j}\stackrel{\text{def}}{:=}(\check{\rho}\bar{\mathcal{H}})_{i-1,j}+
\frac{1}{2}\phi(r_{i-1/2})((\check{\rho} \bar{\mathcal{H}})_{i,j}-(\check{\rho} \bar{\mathcal{H}})_{i-1,j})},
\end{align}
\normalsize
where
\begin{align*}
{ r_{i+1/2} = \frac{(\check{\rho}\tilde{\mathcal{H}})_{i+1,j}-(\check{\rho}\tilde{\mathcal{H}})_{i,j}}
{(\check{\rho} \tilde{\mathcal{H}})_{i,j}-(\check{\rho} \tilde{\mathcal{H}})_{i-1,j}} \quad \text{and} \quad
r_{i-1/2}= \frac{(\check{\rho} \tilde{\mathcal{H}})_{i,j}-(\check{\rho} \tilde{\mathcal{H}})_{i-1,j}}
{(\check{\rho} \tilde{\mathcal{H}})_{i-1,j}-(\check{\rho} \tilde{\mathcal{H}})_{i-2,j}}}.
\end{align*}
In above equation for each $i$,
\begin{align}
\label{eqn66}
 \bar{\mathcal{H}}_{i,j} = \tilde{\mathcal{H}}_{i,j} + \sum_{k}\mathcal{G}_{i,k,j,j}\sigma_{k}
\end{align}
It is observed that the magnitude of the kernel $\mathcal{G}_{i,k,j,j}$ in equation ~\ref{eqn66} diminishes rapidly as distance $|k-i|$
increase and therefore, we avoid unnecessary computation expense by allowing value of $k$ up to three terms.
So updated value of film thickness is rewritten as
\begin{align}
\label{eqn67}
 \bar{\mathcal{H}}_{i,j} = \tilde{\mathcal{H}}_{i,j} + \sum_{k=i-1}^{i+1}\mathcal{G}_{i,k,j,j}\sigma_{k}.
\end{align}
Hence, Eqn.~(\ref{eqn62}) is illustrated as
\begin{align}
\label{eqn68}
{ \mathcal{C}_{i+2,\phi}\sigma_{i+2} +\mathcal{C}_{i+1,\phi}\sigma_{i+1}+\mathcal{C}_{i,\phi}\sigma_{i}
+\mathcal{C}_{i-1,\phi}\sigma_{i-1}+\mathcal{C}_{i-2,\phi}\sigma_{i-2}= R_{i,j,\phi}},
\end{align}
where $R_{i,j,\phi}$ and $\mathcal{C}_{i\pm.,\phi}$ are residual and coefficients of matrix arising due to linearized form involving the limiter function.
This setting leads to a band matrix formulation which is solved using PAQIF algorithm.
\subsubsection{Limiter based Weighted change Newton-Raphson method}\label{subsec:num2}
The underline philosophy of weighted change Newton-Raphson method is more physical than mathematical.
When diffusive coefficient tends to zero, pressure becomes large enough and non local effect of film thickness dominates in the region.
Therefore, even a small deflection in pressure change produces high error in updated film thickness eventually leads blow up the solution
after few iterations. This numerical instability is overcome by interacting with the neighborhood points during iteration.    
During this process the computed change of pressure at one point of the line are shared to its neighbor cells.
In other words, a given point of a line new pressure $\bar{u}_{i,j}$ is computed from the summation of the changes
coming from neighboring points plus the old approximated pressure $\tilde{u}_{i,j}$
\begin{align}
\label{eqn69}
\bar{{u}}_{i,j} = \tilde{{u}}_{i,j}+\sigma_{i,j}-\dfrac{(\sigma_{i+1,j}+\sigma_{i-1,j}+\sigma_{i,j+1}+\sigma_{i,j-1})}{4}
\end{align}
In this case, changes are incorporated only at the end of a complete iteration sweep.
Therefore, overall splitting is derived as below
\begin{align}
\label{eqn70}
{  \frac{\epsilon^{X}_{i+1/2,j}[({u}_{i+1,j} + \sigma_{i+1}-\frac{(\sigma_{i}+\sigma_{i+2})}{4})  
  -({u}_{i,j}+ \sigma_{i}-\frac{(\sigma_{i-1}+\sigma_{i+1})}{4})]}{h_{x}} } \nonumber \\
{ +\frac{\epsilon^{X}_{i-1/2,j}[({u}_{i-1,j} + \sigma_{i-1}-\frac{(\sigma_{i-2}+\sigma_{i})}{4})
 -({u}_{i,j}+ \sigma_{i}-\frac{(\sigma_{i-1}+\sigma_{i+1})}{4})]}{h_{x}} }\nonumber \\
{ +\frac{\epsilon^{Y}_{i,j+1/2}[{u}_{i,j+1}-\frac{\sigma_{i}}{4}-({u}_{i,j}+ \sigma_{i}-\frac{(\sigma_{i-1}+\sigma_{i+1})}{4})]}{h_{y}}+ } \nonumber \\
{ \frac{\epsilon^{Y}_{i,j-1/2}[{u}_{i,j-1}-\frac{\sigma_{i}}{4}-({u}_{i,j}+ \sigma_{i}-\frac{(\sigma_{i-1}+\sigma_{i+1})}{4})]}{h_{y}} }  \nonumber \\
{ -h_{y}((\rho \mathcal{H})^{*}_{i+1/2,j}-(\rho \mathcal{H})^{*}_{i-1/2,j}) = 0 }.
\end{align}
The following notion used in Eqn.~\ref{eqn70} defined as
\begin{align}
\label{eqn71}
 \epsilon^{X}_{i \pm 1/2,j}\stackrel{\text{defn}}{:=} h_{y}\epsilon_{i \pm 1/2,j} \nonumber \\
 \epsilon^{Y}_{i,j \pm 1/2}\stackrel{\text{defn}}{:=} h_{x}\epsilon_{i,j \pm 1/2}
\end{align}
\begin{gather*}
{ \epsilon_{i \pm 1/2,j} = 0.5\Big(\frac{\rho(i\pm 1,j) \mathcal{H}^{3}(i \pm 1,j)}{\eta(i \pm 1,j)\lambda}
 +\frac{\rho(i \pm 1,j) \mathcal{H}^{3}(i \pm 1,j)}{\eta(i \pm 1,j)\lambda}\Big)},\\
{ \epsilon_{i,j \pm 1/2} = 0.5\Big(\frac{\rho(i,j\pm 1) \mathcal{H}^{3}(i,j \pm 1)}{\eta(i,j \pm 1)\lambda}
 +\frac{\rho(i,j \pm 1) \mathcal{H}^{3}(i ,j \pm 1)}{\eta(i \pm 1,j \pm 1)\lambda}\Big)}. 
\end{gather*}
\begin{align}
\label{eqn72}
{(\rho \mathcal{H})^{*}_{i+1/2,j}\stackrel{\text{def}}{:=}(\check{\rho} \bar{\mathcal{H}})_{i,j}+
\frac{1}{2}\phi(r_{i+1/2})((\check{\rho}\bar{\mathcal{H}})_{i+1,j}-(\check{\rho} \bar{\mathcal{H}})_{i,j})}
\end{align}
\begin{align}
\label{eqn73}
{(\rho \mathcal{H})^{*}_{i-1/2,j}\stackrel{\text{def}}{:=}(\check{\rho}\bar{\mathcal{H}})_{i-1,j}+
\frac{1}{2}\phi(r_{i-1/2})((\check{\rho} \bar{\mathcal{H}})_{i,j}-(\check{\rho} \bar{\mathcal{H}})_{i-1,j})},
\end{align}
where
\begin{align*}
{ r_{i+1/2} = \frac{(\check{\rho}\tilde{\mathcal{H}})_{i+1,j}-(\check{\rho}\tilde{\mathcal{H}})_{i,j}}
{(\check{\rho} \tilde{\mathcal{H}})_{i,j}-(\check{\rho} \tilde{\mathcal{H}})_{i-1,j}} \quad \text{and} \quad
r_{i-1/2}= \frac{(\check{\rho} \tilde{\mathcal{H}})_{i,j}-(\check{\rho} \tilde{\mathcal{H}})_{i-1,j}}
{(\check{\rho} \tilde{\mathcal{H}})_{i-1,j}-(\check{\rho} \tilde{\mathcal{H}})_{i-2,j}} }.
\end{align*}
In the above equation, discretization of convection term defined same as x-direction splitting case.
However, due to x-direction weighted change Newton-Raphson splitting, the updated value of the film thickness is described as
\begin{align}
\label{eqn74}
 \bar{\mathcal{H}}_{i,j} = \tilde{\mathcal{H}}_{i,j} + \sum_{k}\sigma \mathcal{G}_{i,k,j,j}\sigma_{k}, 
\end{align}
where $$\sigma \mathcal{G}_{i,i,j,j} = \mathcal{G}_{i,i,j,j}-(\mathcal{G}_{i,i-1,j,j}+\mathcal{G}_{i,i+1,j,j}
+\mathcal{G}_{i,i,j,j-1}+\mathcal{G}_{i,i,j,j+1}).$$
After few manipulation of Eqn.~\ref{eqn70}, we get system of band matrix which is solved using PAQIF approach.\\
The force balance equation is incorporated in our numerical calculation by updating the constant value $\mathcal{H}_{0}$.
The updated value of $\mathcal{H}_{0}$ is performed according to
\begin{align}
\label{eqn75}
 \mathcal{H}_{0} \leftarrow \mathcal{H}_{0}-c\Big( \frac{2\pi}{3}
 -h_{x}h_{y}\sum_{i=1}^{n_{x}} \sum_{j=1}^{n_{y}} {u}_{i,j} \Big),
\end{align}
where $c$ is a relaxation parameter having range between $0.01-0.1$.
\subsection{Convergence criterion of hybrid splitting}
In this section, we give a general criteria for the convergence study of hybrid schemes used in our EHL model problem.
Let us reconsider linear system
$$ L_{\kappa}u=f,$$
where $[L_{\kappa}]_{m\times m}$ a regular matrix (for definition see \cite{Varga}) and $f$ and $u$ are known values.
For applying hybrid splitting in above equation matrix $L_{\kappa}$ is understood as 
$$ L_{\kappa}=L_{\kappa}^{\Omega_{\epsilon}}L_{\kappa}^{\Omega'_{\epsilon}},$$
where $[L_{\kappa}^{\Omega_{\epsilon}}]$ and  $[L_{\kappa}^{\Omega'_{\epsilon}}]$ are regular applied splittings in 
$$\Omega_{\epsilon}=\Big\{(x,y)\Big|\min\Big(\frac{\epsilon(x,y)}{h_{x}},\frac{\epsilon(x,y)}{h_{y}}\Big) > 0.6\Big\}$$ 
and 
$$\Omega'_{\epsilon}=\Big\{(x,y)\Big|\min\Big(\frac{\epsilon(x,y)}{h_{x}},\frac{\epsilon(x,y)}{h_{y}}\Big)\le 0.6\Big\}$$
sub-domains respectively.\\
Now assume that $[L_{\kappa}^{\Omega_{\epsilon}}]$ has the following splitting
$$L_{\kappa}^{\Omega_{\epsilon}}=M_{\kappa}^{\Omega_{\epsilon}}-N_{\kappa}^{\Omega_{\epsilon}},$$
where $M_{\kappa}^{\Omega_{\epsilon}}$ is a regular easily invertible matrix and $N_{\kappa}^{\Omega_{\epsilon}}$ 
is a positive rest matrix. Then our splitting can be defined as 
$$ u^{n+1}_{\Omega_{\epsilon}}=u^{n}_{\Omega_{\epsilon}}-(M_{\kappa}^{\Omega_{\epsilon}})^{-1}(L_{\kappa}^{\Omega_{\epsilon}}-f)$$
Then above iteration will converge for any initial guess $u^{0}$ if following theorem holds 
\begin{theorem}
 Let $L_{\kappa}^{\Omega_{\epsilon}}=M_{\kappa}^{\Omega_{\epsilon}}-N_{\kappa}^{\Omega_{\epsilon}}$ be a regular splitting of matrix $L_{\kappa}^{\Omega_{\epsilon}}$
 and $(L_{\kappa}^{\Omega_{\epsilon}})^{-1} \ge 0$, then we have 
 $$\rho((M_{\kappa}^{\Omega_{\epsilon}})^{-1}N_{\kappa}^{\Omega_{\epsilon}})
 =\frac{\rho((L_{\kappa}^{\Omega_{\epsilon}})^{-1}N_{\kappa}^{\Omega_{\epsilon}})}
 {1+\rho((L_{\kappa}^{\Omega_{\epsilon}})^{-1}N_{\kappa}^{\Omega_{\epsilon}})} < 1$$
\end{theorem}
\begin{proof}
For the proof of this theorem we refer to see Varga \cite{Varga}.
\end{proof}
Now we will prove other part of matrix splitting $L_{\kappa}^{\Omega'_{\epsilon}}$. This part of matrix there is no straightforward 
splitting is available (see \cite{Varga,wittum}). Let $L_{\kappa}^{\Omega'_{\epsilon}}$ is regular,
but dense and the designing suitable splitting in the sense of Varga is complicated.
Suppose if it is possible to construct nonsingular matrix $L^{r}_{\kappa}$  such that equation below
$$L_{\kappa}^{\Omega'_{\epsilon}}L^{r}_{\kappa}=M_{\kappa}^{\Omega'_{\epsilon}}-N_{\kappa}^{\Omega'_{\epsilon}} $$
is easy to solve and we can rewrite splitting as 
$$L_{\kappa}^{\Omega'_{\epsilon}}=(M_{\kappa}^{\Omega'_{\epsilon}}-N_{\kappa}^{\Omega'_{\epsilon}}){L^{r}_{\kappa}}^{-1} $$
Then for above splitting our iteration is denoted as
$$u^{n+1}=u^{n}-L^{r}_{\kappa}(M_{\kappa}^{\Omega'_{\epsilon}})^{-1}(L_{\kappa}^{\Omega'_{\epsilon}}-f)$$
Therefore above iteration will converge for any initial guess if following theorem holds
\begin{theorem}
 Let $(M_{\kappa}^{\Omega'_{\epsilon}}-N_{\kappa}^{\Omega'_{\epsilon}})(L^{r}_{\kappa})^{-1}$ be a regular splitting of matrix
 $L_{\kappa}^{\Omega'_{\epsilon}}$
 and $(L_{\kappa}^{\Omega'_{\epsilon}})^{-1} \ge 0$, then we have 
 $$\rho(L^{r}_{\kappa}(M_{\kappa}^{\Omega'_{\epsilon}})^{-1}N_{\kappa}^{\Omega'_{\epsilon}}(L^{r}_{\kappa})^{-1})
 =\frac{\rho((L_{\kappa}^{\Omega'_{\epsilon}})^{-1}N_{\kappa}^{\Omega'_{\epsilon}}(L^{r}_{\kappa})^{-1})}
 {1+\rho((L_{\kappa}^{\Omega'_{\epsilon}})^{-1}N_{\kappa}^{\Omega'_{\epsilon}}(L^{r}_{\kappa})^{-1})} < 1$$ 
\end{theorem}
\section{Numerical Results}\label{sec:six}
In Section~\ref{sec:three}, we have described TVD implementation for solving a large class problems (that is complementarity problem as well as EHL problems) using PAQIF algorithm. In this section, we  investigate the performance of the mentioned splittings. However, in convection-diffusion problem projection on convex set is not required so in that case we use AQIF algorithm.
For linear case study, we consider analytical solution as $u=x^{4}+y^{4}$ from Oosterlee \cite{Oosterlee}, diffusion coefficient $\epsilon=10^{-6}$ and $\kappa=0.0,1/3,-1.0$.
The Dirichlet boundary is imposed for all test cases on domain $\Omega=\Big\{ (x,y); -1 \le x \le 1,-1 \le y \le 1 \Big\}$.
Numerical tests are performed for the problem given as example~\ref{ex:one} using $Ls0$ splitting, $Ls1$ splitting.
The relative error (in $\mathcal{L}_{\infty},\mathcal{L}_{1},\mathcal{L}_{2}$) are plotted in figures(see fig[5-10]).
%
%
%
\begin{figure}[!htbp]
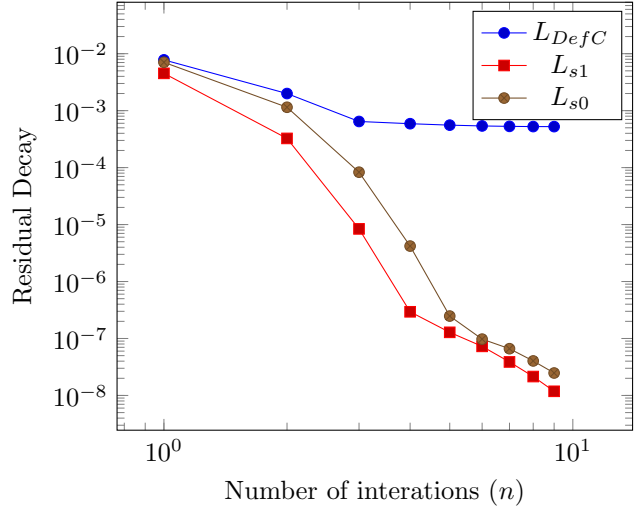

\centering
\tikzpicture
	\axis[
		xlabel= Number of interations ($n$),
		ylabel= Residual Decay,
		xmin=0, xmax = 14,
		ymin=-1.177e-04, ymax = .80e-01,
		ymode=log,
		xmode=log,
		]
	\addplot coordinates {
(1,  7.790e-03)
(2,  1.997e-03)
(3,  6.456e-04)
(4,  5.900e-04)
(5,  5.568e-04)
(6,  5.387e-04)
(7,  5.297e-04)
(8,  5.254e-04)
(9,  5.244e-04)
	};
 	\addplot coordinates {
(1,  4.518e-03)
(2,  3.255e-04)
(3,  8.398e-06)
(4,  2.943e-07)
(5,  1.283e-07)
(6,  7.276e-08)
(7,  3.875e-08)
(8,  2.147e-08)
(9,  1.183e-08)
	};
 	\addplot coordinates {
(1,  7.060e-03)
(2,  1.147e-03)
(3,  8.307e-05)
(4,  4.192e-06)
(5,  2.476e-07)
(6,  9.746e-08)
(7,  6.620e-08)
(8,  4.035e-08)
(9,  2.477e-08)
	};	
	\legend{$L_{DefC}$,$L_{s1}$,$L_{s0}$}
	\endaxis
\endtikzpicture
\caption{Redual Decay of splittings $L_{DefC},L_{s0},L_{s1}$ corresponding value of $\kappa=0.0$}
\label{fig:err}
\end{figure} 
\begin{figure}[!htbp]
\centering
\tikzpicture
	\axis[
		xlabel= Number of interations ($n$),
		ylabel= Residual Decay,
		xmin=0, xmax = 14,
		ymin=-1.177e-04, ymax = .80e-01,
		ymode=log,
		xmode=log,
		]
	\addplot coordinates {
(1,  7.806e-03)
(2,  7.203e-03)
(3,  6.504e-03)
(4,  5.954e-03)
(5,  5.619e-03)
(6,  5.432e-03)
(7,  5.334e-03)
(8,  5.287e-03)
(9,  5.265e-03)
	};
 	\addplot coordinates {
(1,  5.286e-03)
(2,  3.928e-04)
(3,  1.183e-05)
(4,  1.077e-07)
(5,  2.910e-08)
(6,  2.158e-08)
(7,  1.215e-08)
(8,  7.312e-09)
(9,  4.147e-09)
	};
 	\addplot coordinates {
(1,  7.052e-03)
(2,  1.056e-03)
(3,  7.325e-05)
(4,  3.481e-06)
(5,  5.456e-08)
(6,  3.997e-08)
(7,  2.227e-08)
(8,  1.291e-08)
(9,  8.669e-09)
	};	
	\legend{$L_{DefC}$,$L_{s1}$,$L_{s0}$}
	\endaxis
\endtikzpicture
\caption{Redual Decay of splittings $L_{DefC},L_{s0},L_{s1}$ corresponding value of $\kappa=1/3$}
\label{fig:err2}
\end{figure}
\begin{figure}[!htbp]
\centering
\tikzpicture
	\axis[
		xlabel= Mesh size ($N$),
		ylabel= Relative error,
		ymode=log,
		xmode=log,
		]
	\addplot
coordinates {
(16^2,  1.1757906e-02)
(32^2,  1.76038e-03)
(64^2,  2.57573e-04)
(128^2,  3.47087e-05)
(256^2,  4.54820e-06)
(512^2,  6.02630e-07)
	};
	\addplot
coordinates {
(16^2,  2.25826e-03)
(32^2,  3.32640e-04)
(64^2,  4.20422e-05)
(128^2, 5.37451e-06)
(256^2,  6.78313e-07)
(512^2,  8.51091e-08)
	};	
  \addplot
coordinates {
(16^2,  1.83078e-02)
(32^2,  2.72048e-03)
(64^2,  3.43166e-04)
(128^2,  4.37263e-05)
(256^2,  5.51381e-06)
(512^2,  6.91644e-07)
	};  
	
	\legend{$\mathcal{L}_{\infty}$,$\mathcal{L}_{1}$,$\mathcal{L}_{2}$}
	\endaxis
\endtikzpicture
\caption{Relative Error of splittings $L_{s0}$ corresponding value of $\kappa=1/3$}
\label{fig:err3}
\end{figure}
\begin{figure}[!htbp]
\centering
\tikzpicture
	\axis[
		xlabel= Mesh size ($N$),
		ylabel= Relative error,
		ymode=log,
		xmode=log,
		]
	\addplot
coordinates {
(16^2,  1.24738e-02)
(32^2,  2.00172e-03)
(64^2,  5.88728e-04)
(128^2,  1.84579e-04)
(256^2,  5.06126e-05)
(512^2,  1.28329e-05)
	};
	\addplot
coordinates {
(16^2,  1.50246e-03)
(32^2,  1.73122e-04)
(64^2,  6.97083e-05)
(128^2, 2.31011e-05)
(256^2,  6.64633e-06)
(512^2,  1.78059e-06)
	};	
  \addplot
coordinates {
(16^2,  1.36309e-02)
(32^2,  1.68694e-03)
(64^2,  6.33640e-04)
(128^2,  2.08905e-04)
(256^2,  5.93352e-05)
(512^2,  1.57608e-05)
	};  
	
	\legend{$\mathcal{L}_{\infty}$,$\mathcal{L}_{1}$,$\mathcal{L}_{2}$}
	\endaxis
\endtikzpicture
\caption{Relative Error of splittings $L_{s0}$ corresponding value of $\kappa=0.0$}
\label{fig:err4}
\end{figure}
\begin{figure}[!htbp]
\centering
\tikzpicture
	\axis[
		xlabel= Mesh size ($N$),
		ylabel= Relative error,
		ymode=log,
		xmode=log,
		]
	\addplot
coordinates {
(16^2,  1.62417e-02)
(32^2,  1.01696e-02)
(64^2,  3.89903e-03)
(128^2, 1.20459e-03)
(256^2, 3.42856e-04)
(512^2, 9.05700e-05)
	};
	\addplot
coordinates {
(16^2,  2.30732e-03)
(32^2,  1.03223e-03)
(64^2,  3.65800e-04)
(128^2, 1.05973e-04)
(256^2, 2.84576e-05)
(512^2, 7.36748e-06)
	};	
  \addplot
coordinates {
(16^2,  2.04151e-02)
(32^2,  9.53668e-03)
(64^2,  3.32527e-03)
(128^2, 9.49264e-04)
(256^2, 2.52429e-04)
(512^2, 6.49791e-05)
	};  
	
	\legend{$\mathcal{L}_{\infty}$,$\mathcal{L}_{1}$,$\mathcal{L}_{2}$}
	\endaxis
\endtikzpicture
\caption{Relative Error of splittings $L_{s0}$ corresponding value of $\kappa=-1.0$}
\label{fig:err5}
\end{figure}

\begin{figure}[!htbp]
\centering
\tikzpicture
	\axis[
		xlabel= Mesh size ($N$),
		ylabel= Relative error,
		ymode=log,
		xmode=log,
		]
	\addplot
coordinates {
(16^2,  1.1806e-02)
(32^2,  2.6203e-03)
(64^2,  5.70504e-04)
(128^2,  1.0694e-04)
(256^2,  1.9619e-05)
(512^2,  3.4032e-06)
	};
	\addplot
coordinates {
(16^2,  2.2562406e-03)
(32^2,  3.575540e-04)
(64^2,  4.33084e-05)
(128^2, 5.45271e-06)
(256^2,  6.76793e-07)
(512^2,  8.44721e-08)
	};	
  \addplot
coordinates {
(16^2,  1.83208e-02)
(32^2,  2.92872e-03)
(64^2,  3.64904e-04)
(128^2,  4.73857e-05)
(256^2,  6.09179e-06)
(512^2,  7.74616e-07)
	};  
	
	\legend{$\mathcal{L}_{\infty}$,$\mathcal{L}_{1}$,$\mathcal{L}_{2}$}
	\endaxis
\endtikzpicture
\caption{Relative Error of splittings $L_{s1}$ corresponding value of $\kappa=1/3$}
\label{fig:err6}
\end{figure}
\begin{figure}[!htbp]
\centering
\tikzpicture
	\axis[
		xlabel= Mesh size ($N$),
		ylabel= Relative error,
		ymode=log,
		xmode=log,
		]
	\addplot
coordinates {
(16^2,  1.27672e-02)
(32^2,  2.73792e-03)
(64^2,  6.22587e-04)
(128^2, 2.07084e-04)
(256^2, 3.98623e-05)
(512^2, 1.58405e-05)
	};
	\addplot
coordinates {
(16^2,  1.49677e-03)
(32^2,  1.80364e-04)
(64^2,  7.33006e-05)
(128^2, 2.37525e-05)
(256^2, 6.73718e-06)
(512^2, 1.76203e-06)
	};	
  \addplot
coordinates {
(16^2,  1.36680e-02)
(32^2,  1.82037e-03)
(64^2,  6.63061e-04)
(128^2, 2.13343e-04)
(256^2, 5.99206e-05)
(512^2, 1.58361e-05)
	};  
	
	\legend{$\mathcal{L}_{\infty}$,$\mathcal{L}_{1}$,$\mathcal{L}_{2}$}
	\endaxis
\endtikzpicture
\caption{Relative Error of splittings $L_{s1}$ corresponding value of $\kappa=0.0$}
\label{fig:err7}
\end{figure}

\begin{figure}[!htbp]
\centering
\tikzpicture
	\axis[
		xlabel= Mesh size ($N$),
		ylabel= Relative error,
		ymode=log,
		xmode=log,
		]
	\addplot
coordinates {
(16^2,  2.32470e-03)
(32^2,  1.01308e-03)
(64^2,  3.78032e-04)
(128^2, 1.07979e-04)
(256^2, 2.86739e-05)
(512^2, 7.39007e-06)
	};
	\addplot
coordinates {
(16^2,  1.56030e-02)
(32^2,  1.00995e-02)
(64^2,  4.35094e-03)
(128^2, 1.44691e-03)
(256^2, 4.47319e-04)
(512^2, 1.28974e-04)
	};	
  \addplot
coordinates {
(16^2,  2.05497e-02)
(32^2,  9.31777e-03)
(64^2,  3.42296e-03)
(128^2, 9.67504e-04)
(256^2, 2.54980e-04)
(512^2, 6.53620e-05)
	};  
	
	\legend{$\mathcal{L}_{\infty}$,$\mathcal{L}_{1}$,$\mathcal{L}_{2}$}
	\endaxis
\endtikzpicture
\caption{Relative Error of splittings $L_{s1}$ corresponding value of $\kappa=-1.0$}
\label{fig:err8}
\end{figure}
       \begin{figure}
        \centering
        \includegraphics[width=12cm,height=12cm,angle =-90,keepaspectratio]{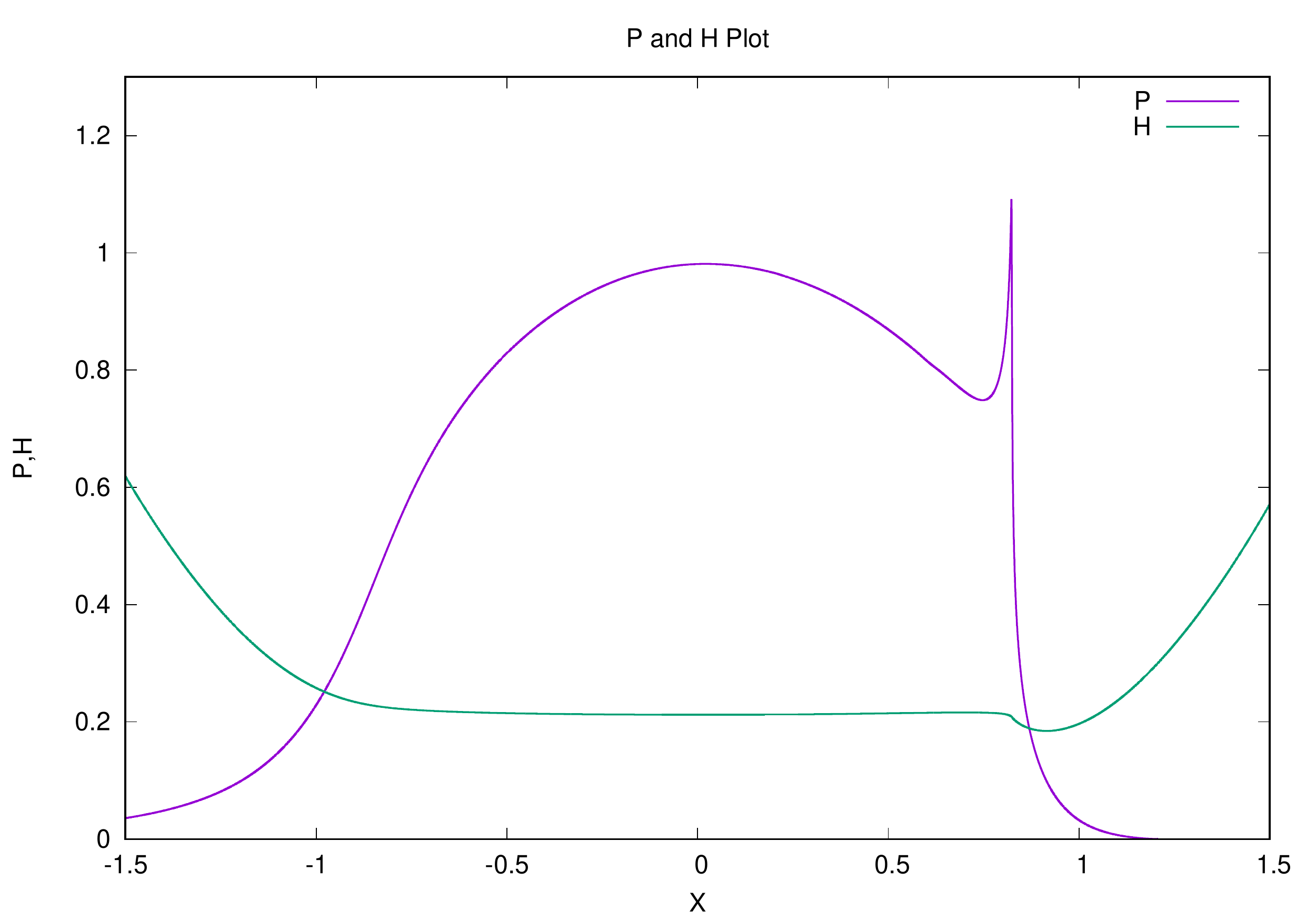}
        \caption{ Steady state line contact case for $G=3500$ , $U=5.5 \times 10^{-11}, W=1.0\times10^{-4}$ $Ls1$ for $\kappa=1/3$}
        \label{figg5}
    \end{figure}
           \begin{figure}
        \centering
        \includegraphics[width=12cm,height=12cm,angle =-90,keepaspectratio]{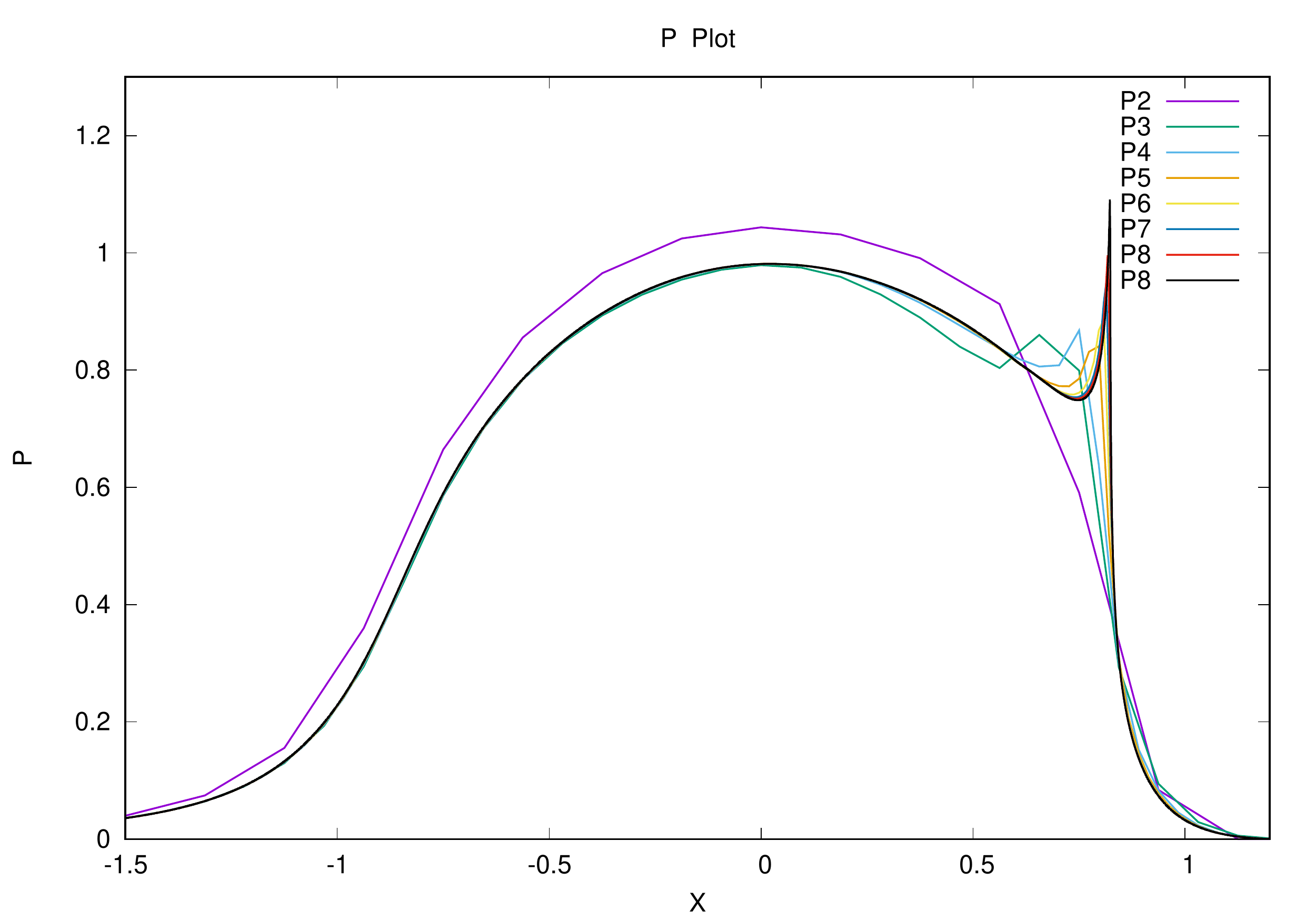}
        \caption{line contact pressure profile plots on grid size $h=32,64,128,256,512,.$}
        \label{figg6}
    \end{figure}
            \begin{figure}
        \centering
        \includegraphics[width=12cm,height=12cm,angle =-90,keepaspectratio]{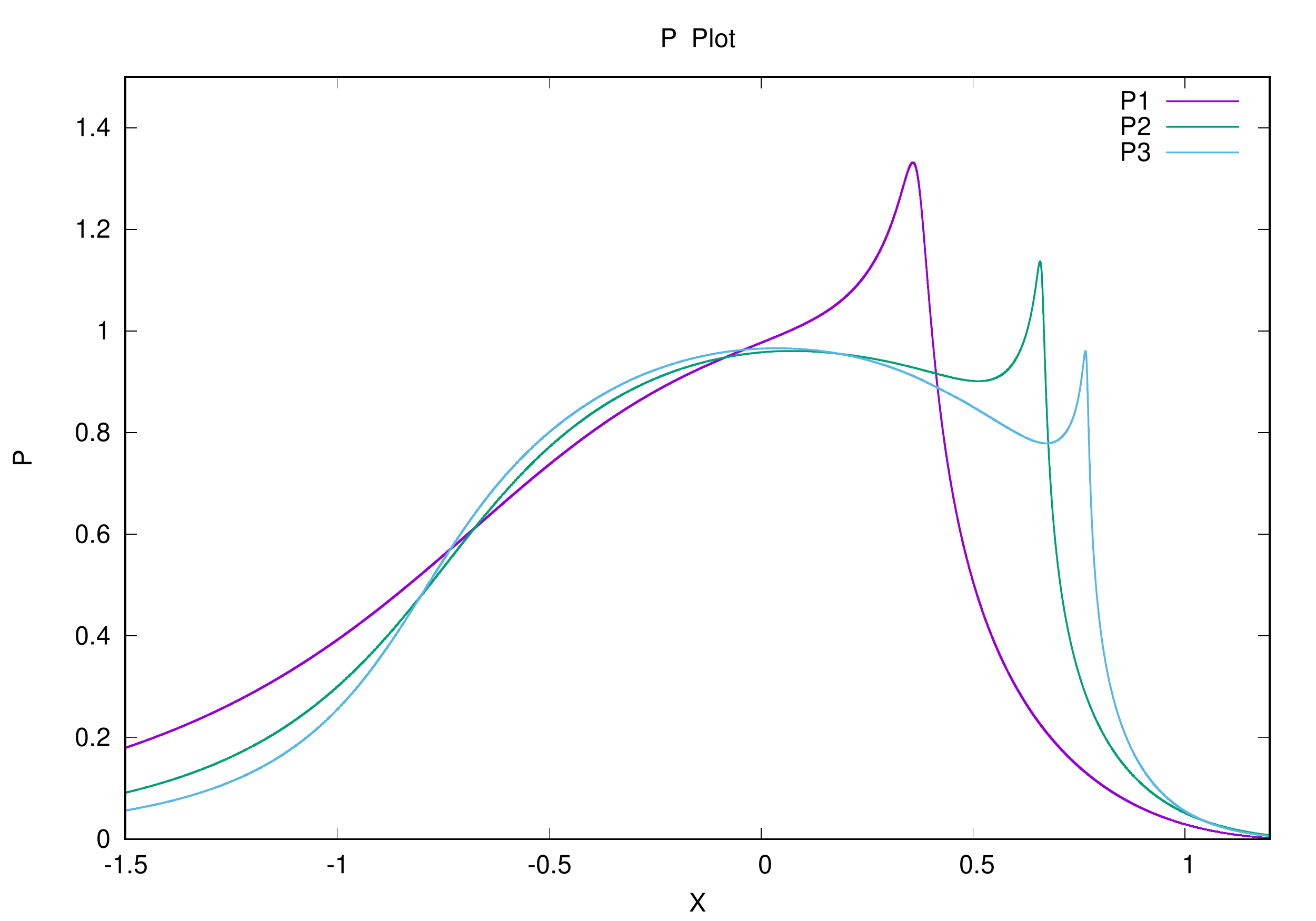}
        \caption{Line contact pressure height on varying load conditions $W=10.0\times10^{-4},W=20.0\times10^{-4},W=30.0\times10^{-4}$}
        \label{figg6}
    \end{figure}
            \begin{figure}
        \centering
        \includegraphics[width=12cm,height=12cm,angle =-90,keepaspectratio]{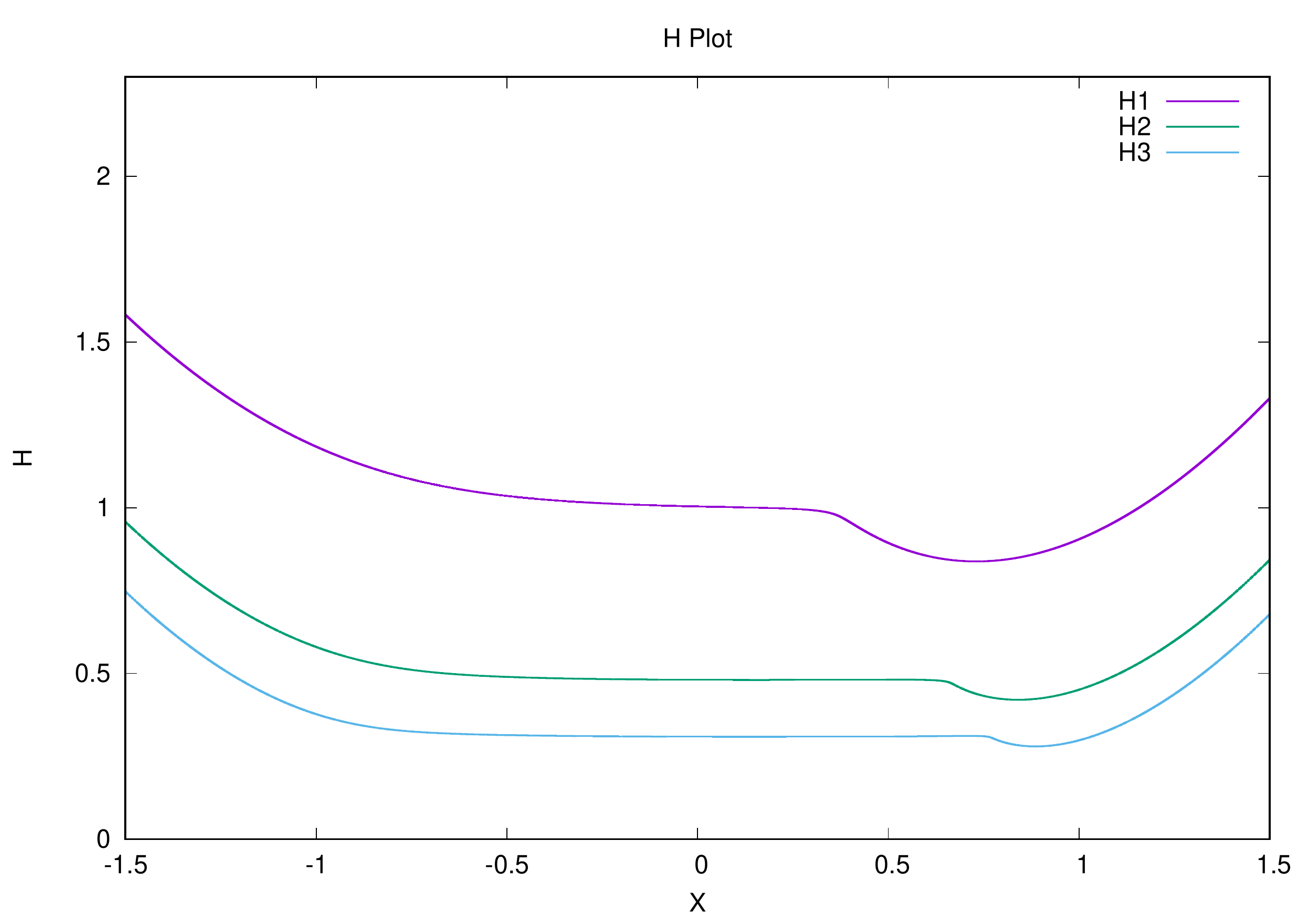}
        \caption{Line contact film thickness $H$ on varying load conditions 
        $W=10.0\times10^{-4},W=20.0\times10^{-4},W=30.0\times10^{-4}$}
        \label{figg6}
    \end{figure}     
    
      \begin{figure}
        \centering
        \includegraphics[width=12cm,height=12cm,angle =-90,keepaspectratio]{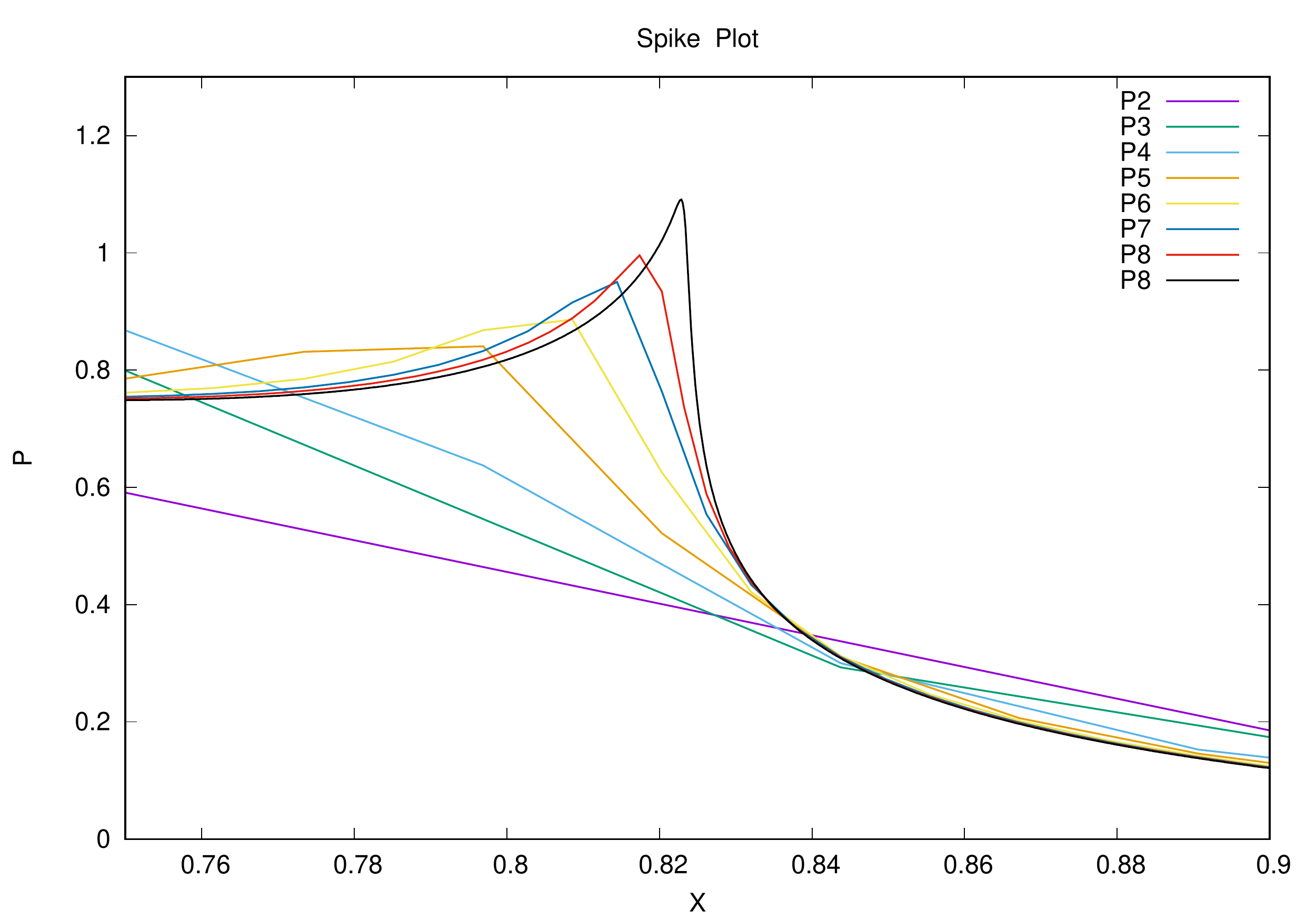}
        \caption{Line contact pressure spike height on different grid size
         $h=32,64,128,256,512,.$}
        \label{figg6}
    \end{figure}
          \begin{figure}
        \centering
        \includegraphics[width=12cm,height=12cm,angle =0,keepaspectratio]{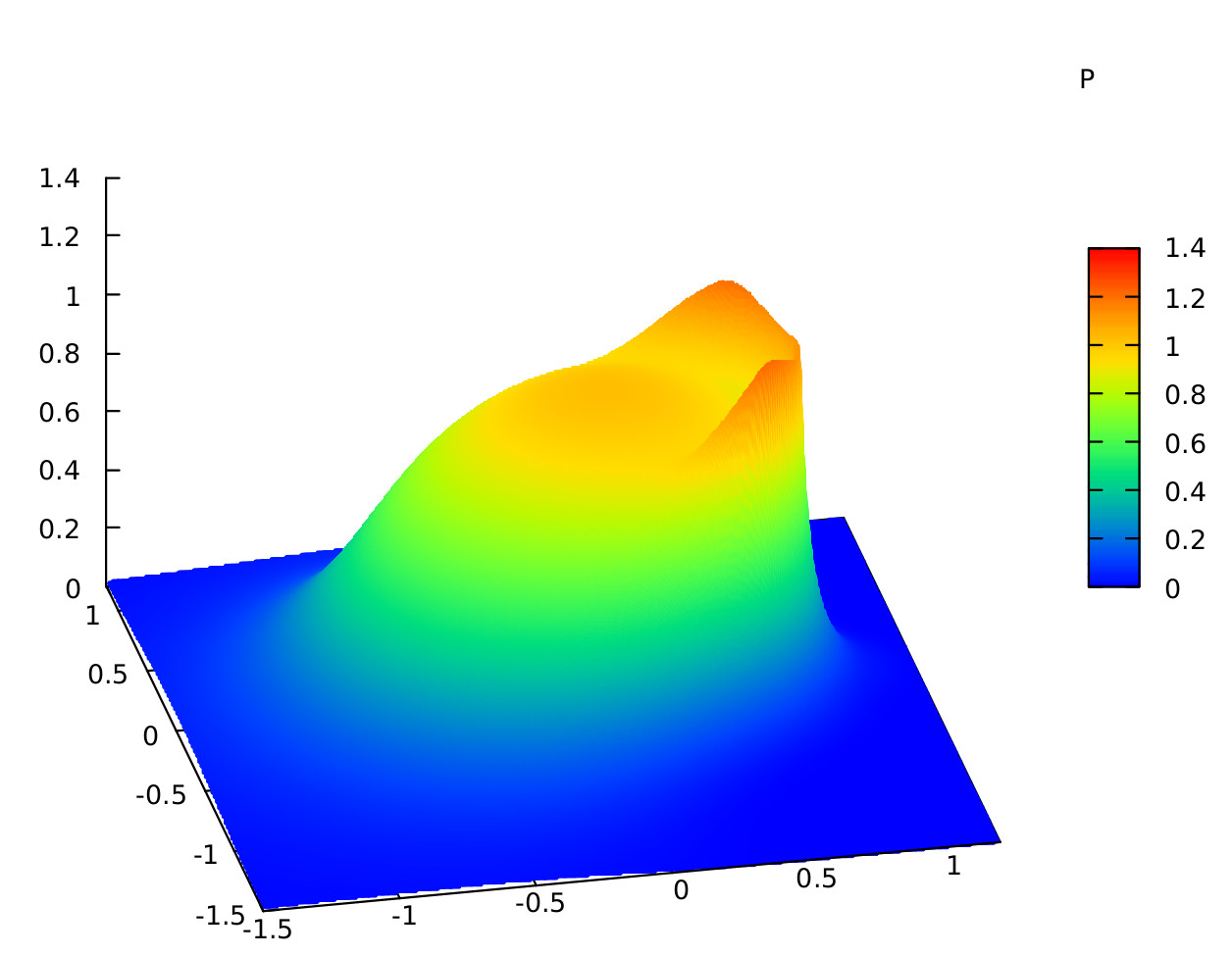}
        \caption{Point contact case EHL solution of pressure profile for $M=20$,$L=10$}
        \label{figg7}
    \end{figure}
          \begin{figure}
        \centering
        \includegraphics[width=12cm,height=12cm,angle =0,keepaspectratio]{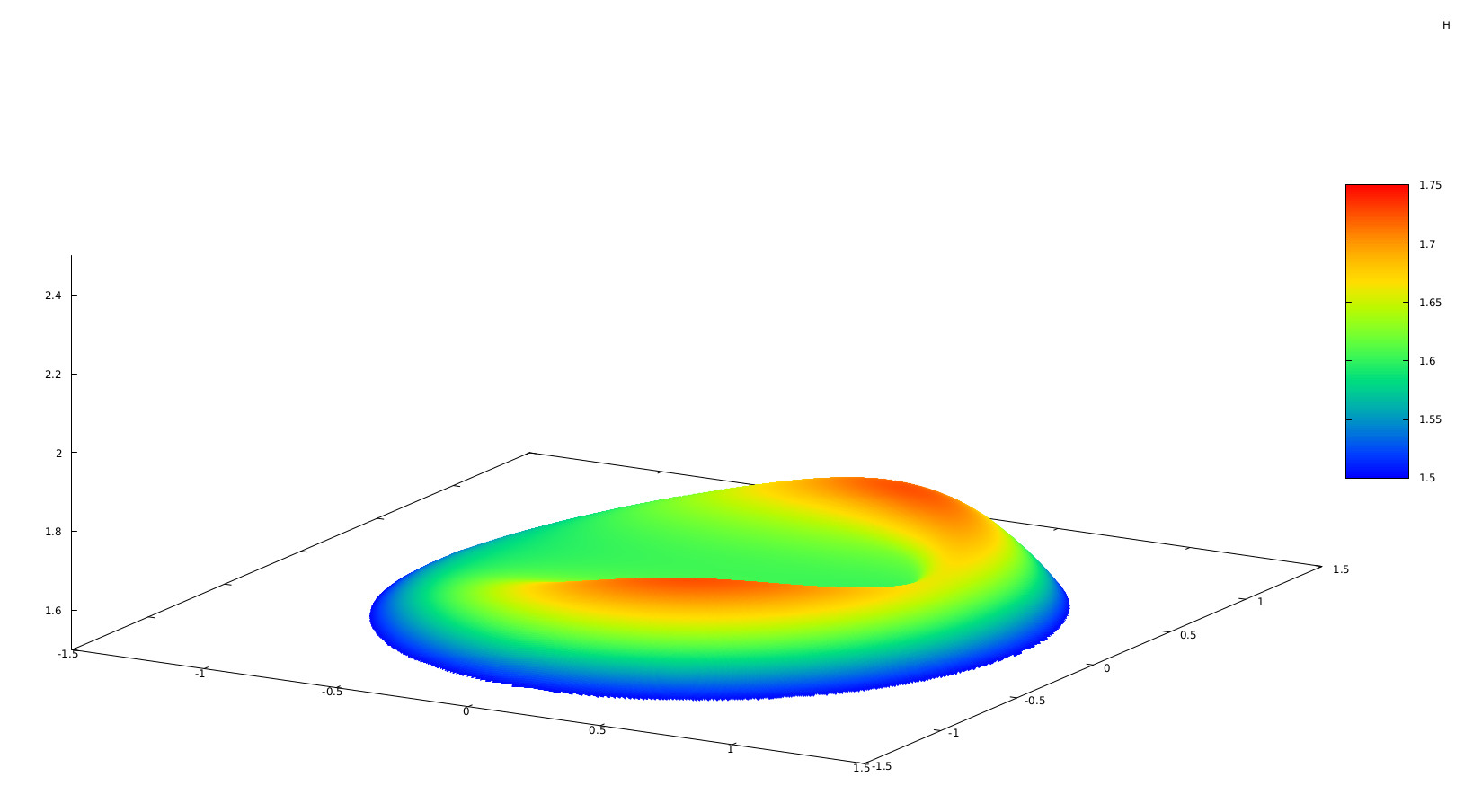}
        \caption{Point contact case inverted film thickness profile for $M=20,L=10$}
        \label{figg8}
    \end{figure}
The $L_{2}$ norm error is evaluated in the following way
\begin{align}
\label{eqn96}
\mathcal{L}_{2}(k,k-1)= \sqrt{H^{d}\sum\Big(\tilde{u}^{k-1}-I_{h}^{H}\bar{u}^{k}\Big)^{2}},
\end {align}
where $H$ is the mesh size on grid $k-1$, $\bar{u}^{k}$ is the converged solution on grid $k$ and $d$ 
denotes the dimension of the problem. The order of convergence is derived as
\begin{align}
\label{eqn97}
p_{2}=\frac{\log \mathcal{L}_{2}(k-1,k-2)-\log \mathcal{L}_{2}(k,k-1)}{\log 2}, 
\end{align}
where $p_{2}$ is the order of discretization in $\mathcal{L}_{2}$ norm.
$\mathcal{L}_{\infty}$ and $\mathcal{L}_{1}$-error are also computed in similar way. From the numerical experiments, we observe that splitting $Ls0$ and $Ls1$ always show fast residual decay compare to classical defect-correction. Fig.~\ref{fig:err} and Fig.~\ref{fig:err2} present the residual decay results for $Ls0$ splitting , $Ls1$ splitting and classical defect-correction technique for $\kappa=0.0,1/3,-1.0$.
Moreover, residual decay of splitting $Ls1$ is more better than splitting $Ls0$. On the other hand, we observe that splitting $Ls0$ has larger range of robustness ($-1.0 \le \kappa \le 0.9$) than splitting $Ls1$ ($-1.0 \le \kappa \le 0.8$).\\
For solving EHL case , we take hertizian pressure distribution as an initial pressure guess.
We perform numerical experiments on EHL model defined in Section~\ref{sec:one}(3.3).
We take  Moes (\cite{moes}) dimensionless parameters $M=20$ and $L=10$.
For the point contact case, a typical pressure profile  and film thickness profile is shown in Fig[16] and Fig[17].
We fix the parameter $\alpha=1.7 \times 10^{-8}$ over computational domain $\Omega=[-2.5,2.5]\times[-2.5,2.5]$. In all cases , we take finer grid points up to $(1024+1)\times (1024+1)$ and coarse grid points up to $(16+1) \times (16+1)$.
Comparisons of relative error in $\mathcal{L}_{2},\mathcal{L}_{1}$ and $\mathcal{L}_{\infty}$ norms between
$\kappa=0.0,1/3,-1.0$ splittings $L_{hs1}$ and $L_{hs2}$ (see section 3.3.2) are performed which are presented in Fig[18-23].
The solutions EHL line contact case (see example 6) is plotted in Fig[11-15].
It is observed that pressure peak get resolved when we increase the grid size of computational domain (see fig.[12,15]). It  is also noted that as load parameter increases then pressure peak height and film thickness get supressed (see Fig[13] and Fig[14]).
\begin{figure}[!htbp]
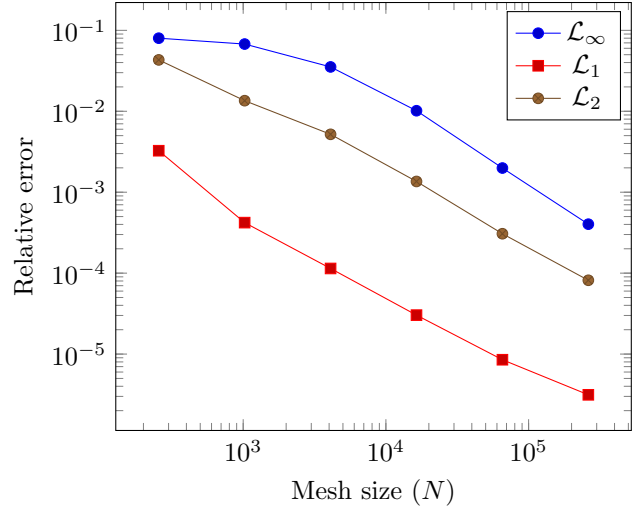

\centering
\tikzpicture
	\axis[
		xlabel= Mesh size ($N$),
		ylabel= Relative error,
		ymode=log,
		xmode=log,
		]
	\addplot
coordinates {
(16^2,  7.99935e-02)
(32^2,  6.76884e-02)
(64^2,  3.53135e-02)
(128^2, 1.01542e-02)
(256^2, 1.98897e-03)
(512^2, 4.02685e-04)
	};
	\addplot
coordinates {
(16^2,  3.25500e-03)
(32^2,  4.20806e-04)
(64^2,  1.14226e-04)
(128^2, 3.02821e-05)
(256^2, 8.51309e-06)
(512^2, 3.13893e-06)
	};	
  \addplot
coordinates {
(16^2,  4.31253e-02)
(32^2,  1.35161e-02)
(64^2,  5.18955e-03)
(128^2, 1.35755e-03)
(256^2, 3.06834e-04)
(512^2, 8.16286e-05)
	};  
	
	\legend{$\mathcal{L}_{\infty}$,$\mathcal{L}_{1}$,$\mathcal{L}_{2}$}
	\endaxis
\endtikzpicture
\caption{Relative Error of EHL case $M=20,L=10$ splittings $L_{hs1}$ corresponding value of $\kappa=0.0$}
\label{fig:err8}
\end{figure}
\begin{figure}[!htbp]
\centering
\tikzpicture
	\axis[
		xlabel= Mesh size ($N$),
		ylabel= Relative error,
		ymode=log,
		xmode=log,
		]
	\addplot
coordinates {
(16^2,  1.28495e-01)
(32^2,  6.61681e-02)
(64^2,  3.34724e-02)
(128^2, 8.88278e-03)
(256^2, 1.64936e-03)
(512^2, 2.79280e-04)
	};
	\addplot
coordinates {
(16^2,  3.46499e-03)
(32^2,  4.17570e-04)
(64^2,  1.07470e-04)
(128^2, 2.70266e-05)
(256^2, 7.15546e-06)
(512^2, 2.77208e-06)
	};	
  \addplot
coordinates {
(16^2,  4.97302e-02)
(32^2,  1.40651e-02)
(64^2,  5.05401e-03)
(128^2, 1.23452e-03)
(256^2, 2.47734e-04)
(512^2, 6.00344e-05)
	};  
	
	\legend{$\mathcal{L}_{\infty}$,$\mathcal{L}_{1}$,$\mathcal{L}_{2}$}
	\endaxis
\endtikzpicture
\caption{Relative Error of EHL case $M=20,L=10$ splittings $L_{hs1}$ corresponding value of $\kappa=1/3$}
\label{fig:err9}
\end{figure}
\begin{figure}[!htbp]
\centering
\tikzpicture
	\axis[
		xlabel= Mesh size ($N$),
		ylabel= Relative error,
		ymode=log,
		xmode=log,
		]
	\addplot
coordinates {
(16^2,  7.50604e-02)
(32^2,  7.55099e-02)
(64^2,  4.53322e-02)
(128^2, 1.61611e-02)
(256^2, 4.50872e-03)
(512^2, 1.10782e-03)
	};
	\addplot
coordinates {
(16^2,  2.97122e-03)
(32^2,  5.91844e-04)
(64^2,  1.91253e-04)
(128^2, 5.75179e-05)
(256^2, 1.67111e-05)
(512^2, 5.21125e-06)
	};	
  \addplot
coordinates {
(16^2,  4.14394e-02)
(32^2,  1.69667e-02)
(64^2,  7.61954e-03)
(128^2, 2.50645e-03)
(256^2, 6.94586e-04)
(512^2, 1.89643e-04)
	};  
	
	\legend{$\mathcal{L}_{\infty}$,$\mathcal{L}_{1}$,$\mathcal{L}_{2}$}
	\endaxis
\endtikzpicture
\caption{Relative Error of EHL case $M=20,L=10$ splittings $L_{hs1}$ corresponding value of $\kappa=-1.0$}
\label{fig:err10}
\end{figure}
\begin{figure}[!htbp]
\centering
\tikzpicture
	\axis[
		xlabel= Mesh size ($N$),
		ylabel= Relative error,
		ymode=log,
		xmode=log,
		]
	\addplot
coordinates {
(16^2,  7.91753e-02)
(32^2,  6.76405e-02)
(64^2,  3.53098e-02)
(128^2, 1.01543e-02)
(256^2, 1.99380e-03)
(512^2, 4.04313e-04)
	};
	\addplot
coordinates {
(16^2,  3.24093e-03)
(32^2,  4.21527e-04)
(64^2,  1.14185e-04)
(128^2, 3.02794e-05)
(256^2, 8.51277e-06)
(512^2, 3.13219e-06)
	};	
  \addplot
coordinates {
(16^2,  4.29201e-02)
(32^2,  1.35422e-02)
(64^2,  5.18823e-03)
(128^2, 1.35750e-03)
(256^2, 3.07193e-04)
(512^2, 8.15121e-05)
	};  
	
	\legend{$\mathcal{L}_{\infty}$,$\mathcal{L}_{1}$,$\mathcal{L}_{2}$}
	\endaxis
\endtikzpicture
\caption{Relative Error of EHL case $M=20,L=10$ splittings $L_{hs2}$ corresponding value of $\kappa=0.0$}
\label{fig:err11}
\end{figure}
\begin{figure}[!htbp]
\centering
\tikzpicture
	\axis[
		xlabel= Mesh size ($N$),
		ylabel= Relative error,
		ymode=log,
		xmode=log,
		]
	\addplot
coordinates {
(16^2,  1.27894e-01)
(32^2,  6.61606e-02)
(64^2,  8.88371e-03)
(128^2, 3.34692e-03)
(256^2, 1.65390e-03)
(512^2, 2.80907e-04)
	};
	\addplot
coordinates {
(16^2,  3.45271e-03)
(32^2,  4.17669e-04)
(64^2,  1.07437e-04)
(128^2, 2.70267e-05)
(256^2, 7.15902e-06)
(512^2, 2.76808e-06)
	};	
  \addplot
coordinates {
(16^2,  4.95561e-02)
(32^2,  1.40784e-02)
(64^2,  5.05304e-03)
(128^2, 1.23467e-03)
(256^2, 2.48217e-04)
(512^2, 5.99858e-05)
	};  
	
	\legend{$\mathcal{L}_{\infty}$,$\mathcal{L}_{1}$,$\mathcal{L}_{2}$}
	\endaxis
\endtikzpicture
\caption{Relative Error of EHL case $M=20,L=10$ splittings $L_{hs2}$ corresponding value of $\kappa=1/3$}
\label{fig:err12}
\end{figure}
\begin{figure}[!htbp]
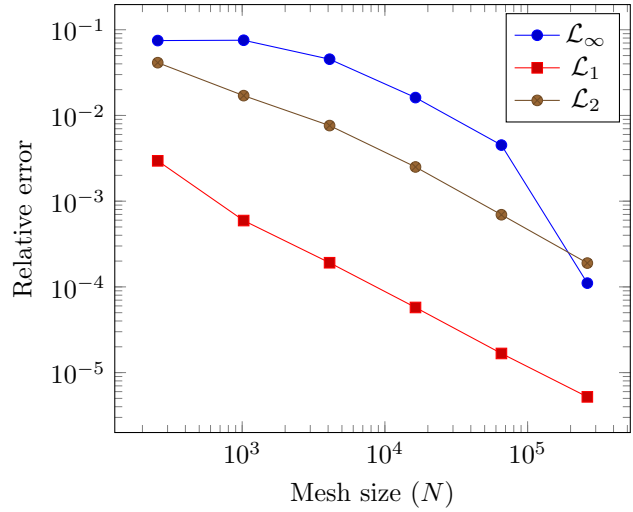

\centering
\tikzpicture
	\axis[
		xlabel= Mesh size ($N$),
		ylabel= Relative error,
		ymode=log,
		xmode=log,
		]
	\addplot
coordinates {
(16^2,  7.47880e-02)
(32^2,  7.54384e-02)
(64^2,  4.53370e-02)
(128^2, 1.61613e-02)
(256^2, 4.51054e-03)
(512^2, 1.10616e-04)
	};
	\addplot
coordinates {
(16^2,  2.95607e-03)
(32^2,  5.94019e-04)
(64^2,  1.91320e-04)
(128^2, 5.75274e-05)
(256^2, 1.67195e-05)
(512^2, 5.21053e-06)
	};	
  \addplot
coordinates {
(16^2,  4.12735e-02)
(32^2,  1.70337e-02)
(64^2,  7.62081e-03)
(128^2, 2.50667e-03)
(256^2, 6.94549e-04)
(512^2, 1.89516e-04)
	};  
	
	\legend{$\mathcal{L}_{\infty}$,$\mathcal{L}_{1}$,$\mathcal{L}_{2}$}
	\endaxis
\endtikzpicture
\caption{Relative Error of EHL case $M=20,L=10$ splittings $L_{hs2}$ corresponding value of $\kappa=-1.0$}
\label{fig:err13}
\end{figure}

%
\section{Conclusion}\label{sec:seven}
A PAQIF/AQIF parallel algorithm is introduced to solve a wider class of problems emerge in linear and nonlinear elliptic PDEs and, complementarity problems (In particular applications in tribology related EHL problems).
The PAQIF algorithm provides the most natural, robust and systematic  way to solve complementarity type problems ( in particular EHL problems) on parallel computers once the Jacobian matrix of discretized system is reasonably approximated into a banded matrix system and then projecting the system solution onto a convex set. In the present work, a detail discussion is carried out to move forward in this direction by giving a class of splitting ( in other word providing a suitable preconditioner for the original discrete problem). A convergence criteria of such approximated splitting is also discribed by giving a mathematical justifications.
The key concept of using the mentioned splitting to accelerate artificial diffusion only in the region of steep gradient of solution profile and to enhance the accuracy on the other portion (smooth region of solution profile) of the domain. 
Additionally, the hybrid line splitting has been designed with help of a switcher which depends upon the magnitude of $\epsilon/h$.
The derived switcher is important entity as it noticeably allows us to resolve the ill-conditioning of the discretized matrix. The robustness of the splittings are interpreted by carrying out a series of numerical experiments.
As an application part, a limiter based direct parallel solver is introduced for solving EHL point and line contact problems in the form of LCP on parallel computers. An accurate pressure profile in EHL model is achieved by sweeping out the iterations in $x$ and $y$ direction alternatively.
Numerical experiments confirm that the performance of direct parallel solvers are robust not only for linear cases but also for EHL models too.
The above treatment can be easily extendable in time dependent EHL as well as Thermo-elastic Lubrication model. 
\section{Acknowledgment}
First author got full support by DST-SERB Project reference no.PDF/2017/000202 under N-PDF fellowship program 
and working group at the Tata Institute of Fundamental Research, TIFR-CAM, Bangalore.
First author is also highly indebted to Prof. Pravir Dutt, IIT Kanpur for fruitful suggestions and guidance during author's IIT Kanpur visit.
\appendix
\section{Some Notation used in EHL model}\label{app:one}
$p_{H} \rightarrow$ Maximum Hertzian pressure.\\
$\eta_{0}\rightarrow$ Ambient pressure viscosity.\\
$H_{00}\rightarrow$  Central offset film thickness.\\
$a\rightarrow$ Radius of point contact circle.\\
$\alpha \rightarrow$ Pressure viscosity coefficient.\\
$u_{s} = u_{1}+u_{2}$, where $u_{1}$ upper surface velocity and $u_{2}$ lower surface velocity respectively.\\
$p_{0} \rightarrow$ Constant ($p_{0}=1.98 \times 10^{8}$), $z$ is pressure viscosity index ($z=0.68$).\\
$R \rightarrow$ Reduced radius of curvature defined as $R^{-1}=R_{1}^{-1} +R_{2}^{-1}$,\\
where $R_{1}$ and $R_{2}$ are curvature of upper contact surface and lower contact surface respectively.\\
$L$ and $M$ are Moes parameters and they are related as below.\\
$ L=G(2U)^{\frac{1}{4}}, M=W(2U)^{-\frac{1}{2}}$, where \\
$2U=\dfrac{(\eta_{0}u_{s} )}{(E^{'}R)}, W=\dfrac{F}{E'R},p_{H}=\dfrac{(3F)}{(2 \pi a^{2})}$.\\
$\sigma^{n+1}=u^{n+1}-u^{n}$ denote as difference between latest approximation solution $u^{n+1}$ and its predecessor $u^{n}$.

%
%



\end{document}